\documentclass{article}

\usepackage[T2A]{fontenc}
\usepackage[utf8]{inputenc}
\usepackage[russian]{babel}
\usepackage{amsfonts}
\usepackage{amsthm}
\usepackage{amsmath}
\usepackage{amssymb}
\usepackage{cancel}
\usepackage{tabu}
\usepackage{array}
\usepackage{tikz}
\usepackage{hyperref}
\usepackage{makecell}
\usepackage{xcolor}
\usepackage{graphicx}
\usepackage[final]{pdfpages}

\graphicspath{ {./images/} }

\title{Эргодичность границы Мартина графа Юнга--Фибоначчи. I}

\author{И. А. Бочков, В. Ю. Евтушевский}

\begin{document}

\maketitle

\tableofcontents

\newpage

\newtheorem{Lemma}{Лемма}

\newtheorem{Alg}{Алгоритм}

\newtheorem{Col}{Следствие}

\newtheorem{theorem}{Теорема}

\newtheorem{Def}{Определение}

\newtheorem{Prop}{Утверждение}

\newtheorem{Problem}{Задача}

\newtheorem{Zam}{Замечание}

\newtheorem{Oboz}{Обозначение}

\newtheorem{Ex}{Пример}

\newtheorem{Nab}{Наблюдение}

\renewcommand{\labelenumi}{\arabic{enumi}$)$}
\renewcommand{\labelenumii}{\arabic{enumi}.\arabic{enumii}$^\circ$}
\renewcommand{\labelenumiii}{\arabic{enumi}.\arabic{enumii}.\arabic{enumiii}$^\circ$}

\section{Введение}

Рассмотрим слова над алфавитом $\{1,2\}$ с данной суммой цифр $n$.
Как известно, их количество есть число Фибоначчи $F_{n+1}$
($F_0=0,$ $F_1=1,$ $F_{k+2}=F_{k+1}+F_k$), и это самая распространённая
комбинаторная
интерпретация чисел Фибоначчи. Также можно думать 
о разбиениях полосы $2\times n$ на домино $1\times 2$ и $2\times 1$,
сопоставляя двойки парам горизонтальных домино, а 
единицы вертикальным домино.

Введём на этом множестве слов частичный порядок: будем говорить,
что слово $x$ предшествует слову $y$, если после удаления 
общего суффикса в слове $y$ остаётся не меньше
двоек, чем в слове $x$ остаётся цифр. 

Это действительно частичный порядок, более того, 
соответствующее частично упорядоченное множество
является модулярной
решёткой, известной как решётка Юнга -- Фибоначчи.

\begin{center}
\includegraphics[width=12cm, height=10cm]{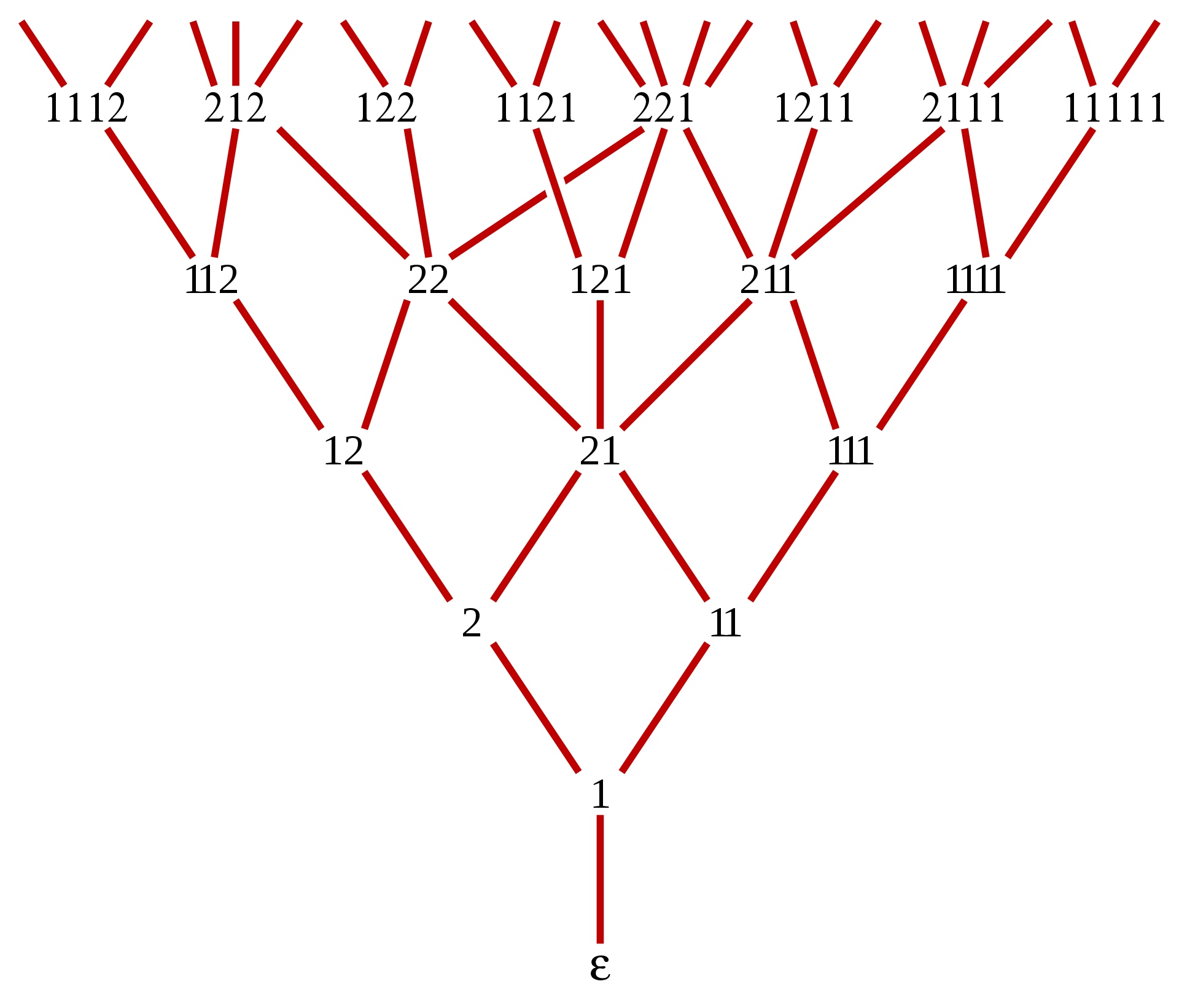}
\end{center}

Графом Юнга -- Фибоначчи (он изображён на рисунке выше) называют диаграмму Хассе этой
решётки. Это градуированный граф, который мы представляем
растущим снизу вверх начиная с пустого слова. 
Градуировкой служит функция суммы цифр. Опишем явно,
как устроены ориентированные
рёбра. Рёбра ``вверх'' из данного
слова $x$ ведут в слова, получаемые из $x$ одной из двух операций:
\begin{enumerate}
    \item  заменить самую левую единицу на двойку;

    \item вставить единицу левее чем самая левая единица.
\end{enumerate}

\renewcommand{\labelenumi}{\arabic{enumi}$^\circ$}
\renewcommand{\labelenumii}{\arabic{enumi}.\arabic{enumii}$^\circ$}
\renewcommand{\labelenumiii}{\arabic{enumi}.\arabic{enumii}.\arabic{enumiii}$^\circ$}

Этот граф помимо модулярности является $1$-дифференциальным, то есть
для каждой вершины исходящая степень на $1$ превосходит 
входящую степень.

Изучение градуированного графа Юнга -- Фибоначчи было инициировано
в 1988 году одновременно и независимо такими математиками, как
Ричард Стенли \cite{St} и Сергей Владимирович Фомин \cite{Fo}.

Причина интереса к нему в том, что существует всего две $1$-дифференциальных
модулярных решётки, вторая --- это решётка диаграмм Юнга, 
имеющая ключевое значение в теории представлений симметрической
группы. 

Центральные вопросы о градуированных графах касаются центральных мер на 
пространстве (бесконечных) путей в графе. Эта точка зрения
последовательно развивалась в работах Анатолия Моисеевича Вершика,
к недавнему обзору которого
\cite{Ver} и приводимой там литературе мы отсылаем читателя.

Среди центральных мер выделяют те, которые являются пределами 
мер, индуцированных путями в далёкие вершины --- так называемую
границу Мартина графа.

Граница пространства путей графа Юнга -- Фибоначчи изучалась в работе
Фредерика Гудмана и Сергея Васильевича Керова (2000) \cite{GK}.

Они использовали алгебраический формализм Окады \cite{Okada}.

Как следует из самого определения, 
асимптотический вопрос о границе напрямую связан с перечислительным
вопросом о числе путей между двумя вершинами графа. 
Отметим важную общую работу С. В. Фомина \cite{Fo1} о перечислении
путей в градуированных графах, в которой приводится
ряд общих тождеств и указывается связь помимо прочего с обобщением
алгоритма Робинсона -- Шенстеда -- Кнута .

Гудман и Керов обходятся без явных формул для числа путей, хотя, как
указал автору Павел Павлович Никитин, из
их рассуждений и можно их извлечь --- но количество слагаемых оказывается
экспоненциальным по длине меньшего из слов. Формула с полиномиальным числом слагаемых была получена в работе  \cite{Evtuh1}, (сокращённая версия которой опубликована как \cite{Evtuh2}). Ниже используются ссылки на оба текста. 

Среди центральных мер на пространстве путей в графе Юнга -- Фибоначчи выделяется так называемая мера Планшереля. Её эргодичность была установлена в работе Керова и Гнедина \cite{KerGned}. 

Целью настоящего цикла из двух работ является доказательство эргодичности остальных мер границы Мартина данного графа, описанных в статье \cite{GK}. Эти меры параметризуются бесконечным словом из единиц и двоек и параметром $\beta\in(0,1]$ (случай $\beta=0$ соответствует мере Планшереля). В данной статье мы сосредотачиваемся на утверждениях, соответствующих случаю $\beta=1$. Основной результат данной работы --- Теорема \ref{t1} и ей Следствия \ref{main1} и \ref{main2}. В следующей работе Следствия \ref{main1} и \ref{main2} используются ``как чёрный ящик'' для завершения доказательства гипотезы эргодичности Керова -- Гудмана.

\newpage

\section{Подготовка к доказательству основной Теоремы}

\begin{Oboz}
Пусть $\mathbb{YF}$ -- это граф Юнга -- Фибоначчи.
\end{Oboz}

\begin{Def}
Пусть $x \in \mathbb{YF}$. Тогда номером $x$ будем называть слово из единиц и двоек, соответствующее вершине $x$.
\end{Def}

\begin{Def}
Пусть $\{\alpha_i\}_{i=1}^\infty\in\{1,2\}^\infty$ -- бесконечная последовательность из единиц и двоек. Этой последовательности сопоставим ``бесконечно удалённую вершину'' графа Юнга -- Фибоначчи с номером $x=\ldots\alpha_2\alpha_1$. 
\end{Def}

\begin{Oboz}
Пусть $\mathbb{YF}_\infty$ -- это множество ``бесконечно удалённых вершин'' графа Юнга--Фибоначчи. 
\end{Oboz}

\begin{Oboz}

$\;$

\begin{itemize}
    \item Пусть $x\in \mathbb{YF},$ $n\in\mathbb{N}_0$, $\{\alpha_i\}_{i=1}^n\in\{1,2\}^n:$ номер $x$ -- это $\alpha_n\ldots\alpha_2\alpha_1$. Тогда будем писать, что $x=\alpha_n\ldots\alpha_2\alpha_1$.
    \item Пусть $x\in \mathbb{YF}_\infty$, $\{\alpha_i\}_{i=1}^\infty\in\{1,2\}^\infty:$ номер $x$ -- это $\ldots\alpha_2\alpha_1$. Тогда будем писать, что $x=\ldots\alpha_2\alpha_1$.
\end{itemize}
\end{Oboz}

\begin{Oboz}

$\;$

\begin{itemize}
    \item Пусть $x \in \mathbb{YF}$. Тогда сумму цифр в номере $x$ обозначим за $|x|$.
    \item Пусть $x \in \mathbb{YF}_\infty$. Тогда скажем, что $|x|=\infty$.
\end{itemize}
\end{Oboz}

\begin{Oboz}
Пусть $n\in\mathbb{N}_0$. Тогда
$$\mathbb{YF}_n:=\left\{v\in\mathbb{YF}:\;|v|=n\right\}.$$
\end{Oboz}

\begin{Zam}

$\;$

\begin{itemize}
    \item Пусть $x\in\mathbb{YF}$. Тогда $|x|$ -- это ранг вершины $x$ в графе Юнга -- Фибоначчи.
    \item Пусть $n\in\mathbb{N}_0$. Тогда $\mathbb{YF}_n$ -- это множество вершин графа Юнга -- Фибоначчи ранга $n$.
\end{itemize}
\end{Zam}

\begin{Oboz} 
Пусть $n,m\in\mathbb{N}_0:$ $m\le n$. Тогда
\begin{itemize}
    \item $$\overline{n}:=\{0,1,\ldots,n\};$$
    \item $$\overline{m,n}:=\{m,m+1\ldots,n\}.$$
\end{itemize}
\end{Oboz}

\begin{Def}\label{vniz}
Пусть $x,y,\{y_i\}_{i=0}^n\in \mathbb{YF}:$ $n\in\mathbb{N}_0$, 
$$y=y_0y_1y_2...y_n=x$$
-- это такой путь в графе Юнга--Фибоначчи, что $\forall i \in \overline{n}$ $$|y_i|=|y|-i.$$
Тогда путь
$$y=y_0y_1y_2...y_n=x$$
назовём $yx$-путём ``вниз'' в $\mathbb{YF}$.
\end{Def}

\begin{Zam}
В Определении \ref{vniz} $n=|y|-|x|$.
\end{Zam}

\begin{Oboz}
Пусть $x,y \in \mathbb{YF}$. Тогда количество $yx$-путей ``вниз'' в $\mathbb{YF}$ будем обозначать как $d(x,y)$.
\end{Oboz}

\begin{Zam}
Пусть $x,y \in \mathbb{YF}:$ $|y|<|x|$. Тогда
$$d(x,y)=0.$$
\end{Zam}

\begin{Oboz}
Пусть $x,y\in \mathbb{YF}$. Тогда множество всех $yx$-путей ``вниз'' обозначим за $T(x,y)$.
\end{Oboz}

\begin{Oboz}
$$T(\mathbb{YF}):=\bigcup_{\left\{(x,y)\in \mathbb{YF}^2\right\}}{T(x,y)}.$$
\end{Oboz}

\begin{Oboz}
Пусть $x,y\in \mathbb{YF},$ $t\in T(\mathbb{YF}):$ $t\in T(x,y)$. Тогда будем обозначать вершины этого пути как 
$$y=t(|y|), \; t(|y|-1),\ldots,\; t(|x|+1),\; t(|x|)=x,$$
а также считать, что если $z\in\left(\mathbb{N}_0\textbackslash\overline{|x|,|y|}\right)$, то $t(z)$ не определено.
\end{Oboz}

\begin{Zam}
Ясно, что $t(z)$ -- это вершина, через которую путь $t$ проходит на уровне $z\in\left(\mathbb{N}_0\textbackslash\overline{|x|,|y|}\right)$.
\end{Zam}

\begin{Oboz}

$\;$

\begin{itemize}
    \item Пусть $x \in \mathbb{YF}$. Тогда количество цифр в номере $x$ обозначим за $\#x$.
    \item Пусть $x \in \mathbb{YF}_\infty$. Тогда скажем, что $\#x=\infty$.
\end{itemize}
\end{Oboz}

\begin{Oboz}
Пусть $x \in \left(\mathbb{YF}\cup \mathbb{YF}_\infty\right)$. Тогда:
\begin{itemize} 
    \item Количество единиц в номере $x$ обозначим за $e(x)$;
    \item Количество двоек в номере $x$ обозначим за $d(x)$.
\end{itemize}
\end{Oboz}

\begin{Zam}
Пусть $x \in \mathbb{YF}_\infty$. Тогда:
\begin{itemize} 
    \item $d(x)$ может быть равно бесконечности;
    \item $e(x)$ может быть равно бесконечности.
\end{itemize}
\end{Zam}

\begin{Zam}
Пусть $x \in \left(\mathbb{YF}\cup \mathbb{YF}_\infty\right)$. Тогда:
\begin{itemize} 
    \item $e(x)+d(x)=\#x;$
    \item $e(x)+2d(x)=|x|;$
    \item $\#x+d(x)=|x|.$
\end{itemize}
\end{Zam}

\begin{Oboz}
$$f(x,y,z): \left\{(x,y,z)\subseteq\mathbb{YF}\times \mathbb{N}_0\times\mathbb{N}_0:\;y\in\overline{|x|},z\in\overline{\#x}\right\}\to\mathbb{R}$$
-- это функция, определённая следующим образом:

При $z=0$:
\begin{itemize}
\item Если $x\in \mathbb{YF}$ представляется в виде $x=\alpha_1...\alpha_m\alpha_{m+1}...\alpha_n$, где $|\alpha_{m+1}...\alpha_n|=y,$ $ \alpha_i \in \{1,2 \}$, то
$$f(x,y,0):=\frac{1}{(\alpha_{m+1})(\alpha_{m+1}+\alpha_{m+2})...(\alpha_{m+1}+...+\alpha_{n})}\cdot(-1)^{n-m}\cdot$$
$$\cdot\frac{1}{(\alpha_m)(\alpha_m+\alpha_{m-1})(\alpha_m+\alpha_{m-1}+\alpha_{m-2})...(\alpha_m+...+\alpha_1)}=$$
$$=\frac{1}{(-\alpha_{m+1})(-\alpha_{m+1}-\alpha_{m+2})...(-\alpha_{m+1}-...-\alpha_{n})}\cdot$$
$$\cdot\frac{1}{(\alpha_m)(\alpha_m+\alpha_{m-1})(\alpha_m+\alpha_{m-1}+\alpha_{m-2})...(\alpha_m+...+\alpha_1)};$$
\item Если $x\in \mathbb{YF}$ не представляется в виде $x=\alpha_1...\alpha_m\alpha_{m+1}...\alpha_n$, где  $|\alpha_{m+1}...\alpha_n|=y,$ $\alpha_i \in \{1,2 \}$, то
$$f(x,y,0)=0.$$
\end{itemize}

При $z>0$ (рекурсивное определение):
\begin{itemize}
    \item Если $y=0$, то
$$f(x1,0,z)=f(x1,0,0);$$
    \item Если $y>0$, то
$$f(x1,y,z)=f(x1,y,0)+f(x,y-1,z-1);$$
    \item 
\begin{equation*}
f(x2,y,z)= 
 \begin{cases}
   $$\frac{f(x11,y,z+1)}{1-y}$$ &\text {если $y \ne 1$}\\
   0 &\text{если $y=1$.}
 \end{cases}
\end{equation*}
\end{itemize}
\end{Oboz}

\begin{Ex} \label{exf1}
Значения $f(x,y,z)$ при всех возможных тройках $(x,y,z)\subseteq\mathbb{YF}\times \mathbb{N}_0\times\mathbb{N}_0:$ $x=21221,$ $y\in\overline{|x|},$ $z=0:$

\begin{itemize}
\item $$f(21221,0,0)=\frac{1}{1\cdot3\cdot5\cdot6\cdot8}=\frac{1}{720};$$
\item $$f(21221,1,0)=\frac{1}{(-1)\cdot2\cdot4\cdot5\cdot7}=-\frac{1}{280};$$
\item $$f(21221,2,0)=0;$$
\item $$f(21221,3,0)=\frac{1}{(-2)\cdot(-3)\cdot2\cdot3\cdot5}=\frac{1}{180};$$
\item $$f(21221,4,0)=0;$$
\item $$f(21221,5,0)=\frac{1}{(-2)\cdot(-4)\cdot(-5)\cdot1\cdot3}=-\frac{1}{120};$$
\item $$f(21221,6,0)=\frac{1}{(-1)\cdot(-3)\cdot(-5)\cdot(-6)\cdot2}=\frac{1}{180};$$
\item $$f(21221,7,0)=0;$$
\item $$f(21221,8,0)=\frac{1}{(-2)\cdot(-3)\cdot(-5)\cdot(-7)\cdot(-8)}=-\frac{1}{1680}.$$
\end{itemize}
\end{Ex}

\begin{Ex} \label{exf2}
Значения $f(x,y,z)$ при всех возможных тройках $(x,y,z)\subseteq\mathbb{YF}\times \mathbb{N}_0\times\mathbb{N}_0:$ $|x|\in\overline{4},$ $y\in\overline{|x|},$ $z\in\overline{\#x}:$

\begin{itemize}
    \item $x=\varepsilon$
\begin{center}
\begin{tabular}{ | m{0.5cm} || m{0.5cm} | } 
  \hline
 & $y=0$ \\
  \hline \hline
$z=0$ & $1$ \\ 
  \hline
\end{tabular}
\end{center}
    \item $x=1$
\begin{center}
\begin{tabular}{ | m{0.5cm} || m{0.5cm} | m{0.5cm} | } 
  \hline
 & $y=0$ & $y=1$ \\
  \hline \hline
$z=0$ & $1$ & $-1$ \\
  \hline
$z=1$ & $1$ & $0$ \\ 
  \hline
\end{tabular}
\end{center}
    \item $x=2$
\begin{center}
\begin{tabular}{ | m{0.5cm} || m{0.5cm} | m{0.5cm} | m{0.5cm} | } 
  \hline
 & $y=0$ & $y=1$ & $y=2$ \\
  \hline \hline
$z=0$ & $\frac{1}{2}$ & $0$ & $-\frac{1}{2}$\\
  \hline
$z=1$ & $\frac{1}{2}$ & $0$ &$-\frac{1}{2}$ \\ 
  \hline
\end{tabular}
\end{center}
    \item $x=11$
\begin{center}
\begin{tabular}{ | m{0.5cm} || m{0.5cm} | m{0.5cm} | m{0.5cm} | } 
  \hline
 & $y=0$ & $y=1$ & $y=2$ \\
  \hline \hline
$z=0$ & $\frac{1}{2}$ & $-1$ & $\frac{1}{2}$\\
  \hline
$z=1$ & $\frac{1}{2}$ & $0$ &$-\frac{1}{2}$ \\ 
  \hline
$z=2$ & $\frac{1}{2}$ & $0$ & $\frac{1}{2}$\\ 
  \hline
\end{tabular}
\end{center}
    \item $x=12$
\begin{center}
\begin{tabular}{ | m{0.5cm} || m{0.5cm} | m{0.5cm} | m{0.5cm} | m{0.5cm} | } 
  \hline
 & $y=0$ & $y=1$ & $y=2$ & $y=3$\\
  \hline \hline
$z=0$ & $\frac{1}{6}$ & $0$ & $-\frac{1}{2}$ &$\frac{1}{3}$\\
  \hline
$z=1$ & $\frac{1}{6}$ & $0$ &$-\frac{1}{2}$ &$\frac{1}{3}$\\ 
  \hline
$z=2$ & $\frac{1}{6}$ & $0$ & $-\frac{1}{2}$ &$-\frac{1}{6}$\\ 
  \hline
\end{tabular}
\end{center}
    \item $x=21$
\begin{center}
\begin{tabular}{ | m{0.5cm} || m{0.5cm} | m{0.5cm} | m{0.5cm} | m{0.5cm} | } 
  \hline
 & $y=0$ & $y=1$ & $y=2$ & $y=3$\\
  \hline \hline
$z=0$ & $\frac{1}{3}$ & $-\frac{1}{2}$ & $0$ &$\frac{1}{6}$\\
  \hline
$z=1$ & $\frac{1}{3}$ & $0$ &$0$ &$-\frac{1}{3}$\\ 
  \hline
$z=2$ & $\frac{1}{3}$ & $0$ & $0$ &$-\frac{1}{3}$\\ 
  \hline
\end{tabular}
\end{center}
    \item $x=111$
\begin{center}
\begin{tabular}{ | m{0.5cm} || m{0.5cm} | m{0.5cm} | m{0.5cm} | m{0.5cm} | } 
  \hline
 & $y=0$ & $y=1$ & $y=2$ & $y=3$\\
  \hline \hline
$z=0$ & $\frac{1}{6}$ & $-\frac{1}{2}$ & $\frac{1}{2}$ &$-\frac{1}{6}$\\
  \hline
$z=1$ & $\frac{1}{6}$ & $0$ &$-\frac{1}{2}$ &$\frac{1}{3}$\\ 
  \hline
$z=2$ & $\frac{1}{6}$ & $0$ & $\frac{1}{2}$ &$-\frac{2}{3}$\\ 
  \hline
$z=3$ & $\frac{1}{6}$ & $0$ & $\frac{1}{2}$ &$\frac{1}{3}$\\ 
  \hline
\end{tabular}
\end{center}
    \item $x=112$
\begin{center}
\begin{tabular}{ | m{0.5cm} || m{0.5cm} | m{0.5cm} | m{0.5cm} | m{0.5cm} | m{0.5cm} | } 
  \hline
 & $y=0$ & $y=1$ & $y=2$ & $y=3$ & $y=4$\\
  \hline \hline
$z=0$ & $\frac{1}{24}$ & $0$ & $-\frac{1}{4}$ &$\frac{1}{3}$ &$-\frac{1}{8}$\\
  \hline
$z=1$ & $\frac{1}{24}$ & $0$ &$-\frac{1}{4}$ &$\frac{1}{3}$ &$-\frac{1}{8}$\\ 
  \hline
$z=2$ & $\frac{1}{24}$ & $0$ & $-\frac{1}{4}$ &$-\frac{1}{6}$ &$\frac{5}{24}$\\ 
  \hline
$z=3$ & $\frac{1}{24}$ & $0$ & $-\frac{1}{4}$ &$-\frac{1}{6}$ &$-\frac{1}{8}$\\ 
  \hline
\end{tabular}
\end{center}
    \item $x=22$
\begin{center}
\begin{tabular}{ | m{0.5cm} || m{0.5cm} | m{0.5cm} | m{0.5cm} | m{0.5cm} | m{0.5cm} | } 
  \hline
 & $y=0$ & $y=1$ & $y=2$ & $y=3$ & $y=4$\\
  \hline \hline
$z=0$ & $\frac{1}{8}$ & $0$ & $-\frac{1}{4}$ & $0$ &$\frac{1}{8}$\\
  \hline
$z=1$ & $\frac{1}{8}$ & $0$ & $-\frac{1}{4}$ & $0$ &$\frac{1}{8}$\\
  \hline
$z=2$ & $\frac{1}{8}$ & $0$ & $-\frac{1}{4}$ & $0$ &$\frac{1}{8}$\\
  \hline
\end{tabular}
\end{center}
\item $x=121$
\begin{center}
\begin{tabular}{ | m{0.5cm} || m{0.5cm} | m{0.5cm} | m{0.5cm} | m{0.5cm} | m{0.5cm} | } 
  \hline
 & $y=0$ & $y=1$ & $y=2$ & $y=3$ & $y=4$\\
  \hline \hline
$z=0$ & $\frac{1}{12}$ & $-\frac{1}{6}$ & $0$ &$\frac{1}{6}$ &$-\frac{1}{12}$\\
  \hline
$z=1$ & $\frac{1}{12}$ & $0$ &$0$ &$-\frac{1}{3}$ &$\frac{1}{4}$\\ 
  \hline
$z=2$ & $\frac{1}{12}$ & $0$ & $0$ &$-\frac{1}{3}$ &$\frac{1}{4}$\\
  \hline
$z=3$ & $\frac{1}{12}$ & $0$ & $0$ &$-\frac{1}{3}$ &$-\frac{1}{4}$\\
  \hline
\end{tabular}
\end{center}
\item $x=211$
\begin{center}
\begin{tabular}{ | m{0.5cm} || m{0.5cm} | m{0.5cm} | m{0.5cm} | m{0.5cm} | m{0.5cm} | } 
  \hline
 & $y=0$ & $y=1$ & $y=2$ & $y=3$ & $y=4$\\
  \hline \hline
$z=0$ & $\frac{1}{8}$ & $-\frac{1}{3}$ & $\frac{1}{4}$ & $0$ &$-\frac{1}{24}$\\
  \hline
$z=1$ & $\frac{1}{8}$ & $0$ &$-\frac{1}{4}$ &$0$ &$\frac{1}{8}$\\ 
  \hline
$z=2$ & $\frac{1}{8}$ & $0$ & $\frac{1}{4}$ &$0$ &$-\frac{3}{8}$\\ 
  \hline
$z=3$ & $\frac{1}{8}$ & $0$ & $\frac{1}{4}$ &$0$ &$-\frac{3}{8}$\\ 
  \hline
\end{tabular}
\end{center}
\item $x=1111$
\begin{center}
\begin{tabular}{ | m{0.5cm} || m{0.5cm} | m{0.5cm} | m{0.5cm} | m{0.5cm} | m{0.5cm} | } 
  \hline
 & $y=0$ & $y=1$ & $y=2$ & $y=3$ & $y=4$\\
  \hline \hline
$z=0$ & $\frac{1}{24}$ & $-\frac{1}{6}$ & $\frac{1}{4}$ &$-\frac{1}{6}$ &$\frac{1}{24}$\\
  \hline
$z=1$ & $\frac{1}{24}$ & $0$ &$-\frac{1}{4}$ &$\frac{1}{3}$ &$-\frac{1}{8}$\\ 
  \hline
$z=2$ & $\frac{1}{24}$ & $0$ & $\frac{1}{4}$ &$-\frac{2}{3}$ &$\frac{3}{8}$\\ 
  \hline
$z=3$ & $\frac{1}{24}$ & $0$ & $\frac{1}{4}$ &$\frac{1}{3}$ &$-\frac{5}{8}$\\ 
  \hline
$z=4$ & $\frac{1}{24}$ & $0$ & $\frac{1}{4}$ &$\frac{1}{3}$ &$\frac{3}{8}$\\ 
  \hline
\end{tabular}
\end{center}
\end{itemize}
\end{Ex}

\begin{Oboz}
$$g(x,y):\left\{(x,y)\in\left(\mathbb{YF}\cup\mathbb{YF}_\infty\right)\times\mathbb{N}:\;y\le d(x)\right\} \to \mathbb{N}$$
-- это функция, определённая следующим образом:

Рассмотрим представление $x\in\mathbb{YF}$ в виде $$x=\ldots 2\underbrace{1\ldots1}_{\beta_m}2\ldots2\underbrace{1\ldots1}_{\beta_1}2\underbrace{1\ldots1}_{\beta_0}$$
и определим:
\begin{itemize}
    \item $g(x,1)=\beta_0+1;$
    \item $g(x,2)=\beta_0+\beta_1+3;$
    \item $\ldots$
    \item $g(x,m)=\beta_0+\ldots+\beta_{m-1}+2m-1;$
    \item $\ldots$.
\end{itemize}

\end{Oboz}

\begin{Oboz}
Пусть $x\in\left(\mathbb{YF}\cup\mathbb{YF}_\infty\right),$ $y \in \mathbb{YF}$. Тогда вершину графа Юнга -- Фибоначчи, номер которой -- это конкатенация номеров $x$ и $y$, обозначим за $xy$. 
\end{Oboz}

\begin{Oboz}
$\;$
\begin{itemize}
    \item Пусть $x,y \in \left(\mathbb{YF}\cup\mathbb{YF}_\infty\right)$, и при этом не выполняется то, что $x=y\in\mathbb{YF}_\infty$. Тогда максимальное $z\in\mathbb{N}_0:$ $\exists x',y'\in \left(\mathbb{YF}\cup\mathbb{YF}_\infty\right),$ $z'\in\mathbb{YF}:$ $ x=x'z'$, $y=y'z'$, $|z'|=z$, обозначим за $h'(x,y)$.
    \item Пусть $x=y\in \mathbb{YF}_\infty$. Тогда будем считать, что $h'(x,y)=\infty$.
\end{itemize}
\end{Oboz}

\begin{Oboz}
$\;$
\begin{itemize}
    \item Пусть $x,y \in \left(\mathbb{YF}\cup\mathbb{YF}_\infty\right)$, и при этом не выполняется то, что $x=y\in\mathbb{YF}_\infty$. Тогда максимальное $z\in\mathbb{N}_0:$ $\exists x',y'\in \left(\mathbb{YF}\cup\mathbb{YF}_\infty\right),$ $z'\in\mathbb{YF}:$ $x=x'z'$, $y=y'z'$, $\#z'=z$, обозначим за $h(x,y)$.
    \item Пусть $x=y\in \mathbb{YF}_\infty$. Тогда будем считать, что $h(x,y)=\infty$.
\end{itemize}
\end{Oboz}

\begin{Zam}

Пусть $x,y \in \left(\mathbb{YF}\cup\mathbb{YF}_\infty\right)$. Тогда 

\begin{itemize}
    \item $h'(x,y)$ -- это сумма цифр в самом длинном общем суффиксе номеров $x$ и $y;$
    \item $h(x,y)$ -- это количество цифр в самом длинном общем суффиксе номеров $x$ и $y;$
    \item $h'(x,y)=h'(y,x);$
    \item $h(x,y)=h(y,x)$.
\end{itemize}
\end{Zam}

\begin{theorem}[Теорема 1\cite{Evtuh1}, Теорема 1\cite{Evtuh2}] \label{evtuh}
Пусть $x,y \in \mathbb{YF}:$ $|y|\ge|x|$. Тогда
$$d(x,y)=\sum_{i=0}^{|x|}\left( {f\left(x,i,h(x,y)\right)}\prod_{j=1}^{d(y)}\left(g\left(y,j\right)-i\right)\right).$$ 
\end{theorem}
\begin{Col}
Пусть $y \in \mathbb{YF}$. Тогда
$$d(\varepsilon,y)=\prod_{j=1}^{d(y)}g\left(y,j\right).$$ 
\end{Col}
\begin{proof}
$$d(\varepsilon,y)=\left(\text{По Теореме \ref{evtuh}}\right)=\sum_{i=0}^{|\varepsilon|}\left( {f\left(\varepsilon,i,h(\varepsilon,y)\right)}\prod_{j=1}^{d(y)}\left(g\left(y,j\right)-i\right)\right)=$$ 
$$=\sum_{i=0}^{0}\left( {f\left(\varepsilon,i,0\right)}\prod_{j=1}^{d(y)}\left(g\left(y,j\right)-i\right)\right)= {f\left(\varepsilon,0,0\right)}\prod_{j=1}^{d(y)}\left(g\left(y,j\right)-0\right)= \prod_{j=1}^{d(y)}g\left(y,j\right).$$ 

\end{proof}

\begin{Oboz}
Пусть $w\in \mathbb{YF}_{\infty}$, $m\in\mathbb{N}_0$. Тогда за $w_m\in\mathbb{YF}$ обозначим такую вершину графа Юнга--Фибоначчи, что $\#w_m=m$ и $\exists w'\in\mathbb{YF}_\infty:$ $w=w'w_m$.
\end{Oboz}

\begin{Zam} 

\;

\begin{itemize}
    \item Очевидно, что для любых $w\in \mathbb{YF}_{\infty}$ и $m\in\mathbb{N}_0$ такая вершина существует и однозначно определена.
    \item Ясно, что это просто вершина, номер которой -- это последние (то есть самые правые) $m$ символов номера $w$.
\end{itemize}
\end{Zam}

\begin{Oboz}
Пусть $w\in\mathbb{YF}_\infty$, $v\in\mathbb{YF}$, $m\in\mathbb{N}_0$. Тогда
    $$\mu_w(v,m):= \frac{d(\varepsilon,v)d(v,w_m)}{d(\varepsilon,w_m)}. $$
\end{Oboz}

\begin{Prop}
Пусть $w\in\mathbb{YF}_\infty,$ $v\in\mathbb{YF}$. Тогда существует предел
$$\lim_{m \to \infty}\mu_w(v,m)=\lim_{m \to \infty} \frac{d(\varepsilon,v)d(v,w_m)}{d(\varepsilon,w_m)}.$$
\end{Prop}
\begin{proof}

Ясно, что если $m\ge|v|,$ то $|w_m|\ge m\ge |v|$, а значит, по Теореме \ref{evtuh} при $v,w_m\in\mathbb{YF}$
$$\lim_{m\to\infty}\mu_{w}(v,m
)=\lim_{m\to\infty}\frac{d(\varepsilon,v)d(v,w_m)}{d(\varepsilon,w_m)}
=$$
$$=\lim_{m\to\infty}\frac{\displaystyle d(\varepsilon,v) \sum_{i=0}^{|v|}\left( {f\left(v,i,h(v,w_m)\right)}\prod_{j=1}^{d(w_m)}{\left(g\left(w_m,j\right)-i\right)}\right)}{\displaystyle\prod_{j=1}^{d(w_m)}{g\left(w_m,j\right)}}=$$
$$=d(\varepsilon,v)\lim_{m\to\infty}\left(\sum_{i=0}^{|v|}\left( {f\left(v,i,h(v,w_m)\right)}\prod_{j=1}^{d(w_m)}\frac{\left(g\left(w_m,j\right)-i\right)}{g\left(w_m,j\right)}\right)\right).$$

Ясно, что если $m\ge |v|$, то $|w_m|\ge m\ge |v|$, а значит $h(v,w_m)=h(v,w)$, из чего следует, что наше выражение равняется следующему:
$$d(\varepsilon,v)\lim_{m\to\infty}\left(\sum_{i=0}^{|v|}\left( {f\left(v,i,h(v,w)\right)}\prod_{j=1}^{d(w_m)}\frac{\left(g\left(w_m,j\right)-i\right)}{g\left(w_m,j\right)}\right)\right)=$$
$$=d(\varepsilon,v)\sum_{i=0}^{|v|}\left( {f\left(v,i,h(v,w)\right)}\cdot\lim_{m\to\infty}\prod_{j=1}^{d(w_m)}\frac{\left(g\left(w_m,j\right)-i\right)}{g\left(w_m,j\right)}\right)=$$
$$=(\text{По определению функции $g$})=$$
$$=d(\varepsilon,v)\sum_{i=0}^{|v|}\left( {f\left(v,i,h(v,w)\right)}\prod_{j=1}^{d(w)}\frac{\left(g\left(w,j\right)-i\right)}{g\left(w,j\right)}\right).$$
Из определения функции $g$ очевидно, что данное выражение определено.

\end{proof}
\begin{Oboz}
Пусть $w\in\mathbb{YF}_\infty$, $v\in\mathbb{YF}$. Тогда
    $$\mu_w(v):=\lim_{m \to \infty}\mu_w(v,m)=\lim_{m \to \infty} \frac{d(\varepsilon,v)d(v,w_m)}{d(\varepsilon,w_m)}. $$
\end{Oboz}

\begin{Zam}
Пусть $w\in\mathbb{YF}_\infty$, $v\in\mathbb{YF}$, $m,n\in\mathbb{N}_0:$ $|w_m|\ge |v|=n$. Тогда
\begin{itemize}
    \item $$ {d(\varepsilon,v)d(v,w_m)}={|\left\{t\in T(\varepsilon,w_m):\;t(n)=v\right\}|} ;$$

    \item $$\mu_w(v,m)= \frac{d(\varepsilon,v)d(v,w_m)}{d(\varepsilon,w_m)}=\frac{|\left\{t\in T(\varepsilon,w_m):\;t(n)=v\right\}|}{|\left\{T(\varepsilon,w_m)\right\}|} .$$
    
\end{itemize}
\end{Zam}

\begin{Prop}\label{mera}
Пусть $w\in\mathbb{YF}_\infty$, $m,n\in\mathbb{N}_0:$ $|w_m|\ge n$. Тогда
$$\sum_{v\in\mathbb{YF}_n}\mu_w(v,m)= \sum_{v\in\mathbb{YF}_n}\frac{d(\varepsilon,v)d(v,w_m)}{d(\varepsilon,w_m)}=1. $$
\end{Prop}
\begin{proof}
$$\sum_{v\in\mathbb{YF}_n}\mu_w(v,m)= \sum_{v\in\mathbb{YF}_n}\frac{d(\varepsilon,v)d(v,w_m)}{d(\varepsilon,w_m)}=(\text{так как $|w_m|\ge n$})=\sum_{v\in\mathbb{YF}_n}\frac{|\left\{t\in T(\varepsilon,w_m):\;t(n)=v\right\}|}{|\left\{T(\varepsilon,w_m)\right\}|}=$$
$$=\frac{\displaystyle\sum_{v\in\mathbb{YF}_n}\left|\left\{t\in T(\varepsilon,w_m):\;t(n)=v\right\}\right|}{\displaystyle|\left\{T(\varepsilon,w_m)\right\}|} =\frac{\displaystyle|\left\{T(\varepsilon,w_m)\right\}|}{\displaystyle|\left\{T(\varepsilon,w_m)\right\}|}=1. $$
\end{proof}

\begin{Col} \label{mera1}
Пусть $w\in\mathbb{YF}_\infty$, $n\in\mathbb{N}_0$. Тогда
$$\sum_{v\in\mathbb{YF}_n}\mu_w(v)=1.$$
\end{Col}
\begin{proof}
$$\sum_{v\in\mathbb{YF}_n}\mu_w(v)=\sum_{v\in\mathbb{YF}_n}\left(\lim_{m \to \infty} \frac{d(\varepsilon,v)d(v,w_m)}{d(\varepsilon,w_m)}\right)=\lim_{m \to \infty}\left(\sum_{v\in\mathbb{YF}_n}\left( \frac{d(\varepsilon,v)d(v,w_m)}{d(\varepsilon,w_m)}\right)\right)=$$
$$=(\text{Так как если $m\ge n$, то $|w_m|\ge m\ge n$ })=1.$$
\end{proof}

\begin{Oboz}
Пусть $w\in \left(\mathbb{YF}\cup\mathbb{YF}_\infty\right)$. Тогда 
$$\pi(w):=\prod_{i:\;g(w,i)> 1  } \frac{g(w,i)-1}{g(w,i)}.$$
\end{Oboz}

\begin{Zam} \label{promezhutok}
Ясно, что 
\begin{itemize}
    \item Если $w\in\mathbb{YF}$, то 
$$\pi(w)\in(0,1];$$
    \item Если $w\in\mathbb{YF}_\infty$, то 
$$\pi(w)\in[0,1].$$
\end{itemize}
\end{Zam}

\begin{Oboz}
$$\mathbb{YF}_\infty^+:=\{w\in\mathbb{YF}_\infty:\;\pi(w)>0\}.$$
\end{Oboz}
\begin{Zam}
Ясно, что
$$\mathbb{YF}_\infty^+=\{w\in\mathbb{YF}_\infty:\;\pi(w)\in(0,1)\}.$$
\end{Zam}

\newpage

\section{Доказательство основной Теоремы}

\begin{Oboz}
Пусть $w\in\mathbb{YF}_\infty$, $n\in\mathbb{N}_0$, $\delta\in(0,1)$. Тогда
\begin{itemize}
    \item $$P(w,n,\delta):=\left\{v\in\mathbb{YF}_n:\; {h'(v,w)}\ge (1 - \delta)n\right\};$$
    \item $$ \overline{P}(w,n,\delta):=\left\{v\in\mathbb{YF}_n:\; {h'(v,w)}< (1 - \delta)n\right\}.$$
\end{itemize}
\end{Oboz}

\renewcommand{\labelenumi}{\arabic{enumi}$)$}
\renewcommand{\labelenumii}{\arabic{enumi}.\arabic{enumii}$^\circ$}
\renewcommand{\labelenumiii}{\arabic{enumi}.\arabic{enumii}.\arabic{enumiii}$^\circ$}

\begin{theorem} \label{t1}
Пусть $w\in\mathbb{YF}_\infty^+,$ $\delta\in(0,1)$. Тогда
\begin{enumerate}
    \item $$\lim_{n \to \infty}{\sum_{v\in \overline{P}(w,n,\delta)}\mu_w(v)=0};$$
    \item $$\lim_{n \to \infty}{\sum_{v\in P(w,n,\delta)}\mu_w(v)=1}.$$
\end{enumerate}
\end{theorem}
\begin{proof}

\renewcommand{\labelenumi}{\arabic{enumi}$^\circ$}
\renewcommand{\labelenumii}{\arabic{enumi}.\arabic{enumii}$^\circ$}
\renewcommand{\labelenumiii}{\arabic{enumi}.\arabic{enumii}.\arabic{enumiii}$^\circ$}

\begin{Oboz}
Пусть $w\in\mathbb{YF}_\infty$, $\delta\in(0,1)$, $n,m\in\mathbb{N}_0:$ $m\ge n$. Тогда
$$\overline{P}(w,n,\delta,m):=\left\{(a,b,k)\in\mathbb{YF}\times\mathbb{YF}\times \mathbb{N}:\; ab=w_m, \; |b|< (1-\delta)n,\; k>\frac{\delta n}{2}\right\}.$$
    
\end{Oboz}

\begin{Oboz}
Пусть $v\in\mathbb{YF},$ $n\in\mathbb{N}_0$. Тогда 
$$v^n:=\underbrace {v\ldots v}_{n} .$$
\end{Oboz}

\begin{Lemma}\label{ass}
Пусть $w\in\mathbb{YF}_\infty$, $\delta\in(0,1)$, $n,m\in\mathbb{N}_0$: $m\ge n$. Тогда
$${\sum_{v\in\overline{P}(w,n,\delta)}\frac{d(\varepsilon,v)d(v,w_m)}{d(\varepsilon,w_m)}}\le \sum_{(a,b,k)\in\overline{P}(w,n,\delta,m)}{\frac{d(\varepsilon,2^kb)d(2^k,a)}{d(\varepsilon,w_m)}}.$$

\end{Lemma}
\begin{proof}
$${\sum_{v\in\overline{P}(w,n,\delta)}\frac{d(\varepsilon,v)d(v,w_m)}{d(\varepsilon,w_m)}}\le \sum_{(a,b,k)\in\overline{P}(w,n,\delta,m)}{\frac{d(\varepsilon,2^kb)d(2^k,a)}{d(\varepsilon,w_m)}}\Longleftrightarrow$$
$$\Longleftrightarrow {\sum_{v\in\overline{P}(w,n,\delta)}{\left(d(\varepsilon,v)d(v,w_m)\right)}}\le \sum_{(a,b,k)\in\overline{P}(w,n,\delta,m)}{{\left(d(\varepsilon,2^kb)d(2^k,a)\right)}}\Longleftrightarrow$$
$$\Longleftrightarrow(\text{Так как если $m\ge n$, то $|w_m|\ge m\ge n$ })\Longleftrightarrow$$
$$\Longleftrightarrow{\sum_{v\in\overline{P}(w,n,\delta)}{|\{t\in T(\varepsilon,w_m):\;t(n)=v\}|}}\le \sum_{(a,b,k)\in\overline{P}(w,n,\delta,m)}{{\left(d(\varepsilon,2^kb)d(2^k,a)\right)}}.$$

\begin{Oboz} \label{w}

\;

Пусть $x,y\in \mathbb{YF},$ $t\in T(\mathbb{YF}):$ $t\in T(x,y)$. Тогда обозначим за $c(t)$ самую верхнюю вершину пути $t$ из тех вершин пути $t$, у которых самый короткий общий суффикс с $y$.

Формально:

Пусть $x,y\in \mathbb{YF},$ $t\in T(\mathbb{YF}):$ $t\in T(x,y)$. Тогда обозначим за $c(t)$ такую вершину $t(z)$ пути $t$ $(z\in\overline{|x|,|y|})$, 
что $\nexists z' \in\overline{|x|,|y|}:$ $h'(t(z'),y)<h'(t(z),y)$ и $\nexists z' \in\overline{z+1,|y|}:$ $h'(t(z'),y)=h'(t(z),y)$.
\end{Oboz}

\begin{Zam}
Очевидно, что если $x,y\in \mathbb{YF},$ $t\in T(\mathbb{YF}):$ $t\in T(x,y)$, то вершина $c(t)$ существует и однозначно определена.
\end{Zam}


Давайте рассмотрим теперь путь $t\in T(\varepsilon,w_m)$ такой, что $t(n)\in \overline{P}(w,n,\delta)$.

Пусть $t'$ -- это часть данного пути, которая проходит от вершины $w_m$ до вершины $t(n)$.

Формально:

$t'\in T(t(n),w_m)$, $t'(|w_m|)=t(|w_m|)=w_m$, $t'(|w_m|-1)=t(|w_m|-1)$, $\ldots$, $t'(n+1)=t(n+1)$, $t'(n)=t(n)$.

Рассмотрим вершину $c(t')$. Рассмотрим три случая:
\begin{enumerate}
    \item $\exists x,y,z \in\mathbb{YF}:$ $w_m=x2y$, $c(t')=z1y$. 
    
    Как в данном случае проходит путь $t'$?
    
    \begin{itemize}
        \item По Обозначению \ref{w} $h'(t'(|c(t')|+1),w_m)>h'\left(t'(|c(t')|),w_m\right)=h'(c(t'),w_m)=|y|$. А значит $\exists q\in\mathbb{YF}:$ $t'(|c(t')|+1)=q2y$. Таким образом, между  $t'(|c(t')|+1)$ и $t'(|c(t')|)=c(t')$ шаг ``вниз'' выглядит следующим образом: 
        $$q2y\to z1y.$$
        
        Вспомним, как выглядят родители вершины графа Юнга -- Фибоначчи:
        
        \setlength{\unitlength}{0.20mm}
        \begin{picture}(400,200)
        \put(270,160){$2221\ldots$}
        \put(300,150){\vector(2,-1){200}}
        \put(300,150){\vector(1,-2){50}}
        \put(300,150){\vector(-1,-2){50}}
        \put(300,150){\vector(-2,-1){200}}
        \put(60,30){$1221\ldots$}
        \put(220,30){$2121\ldots$}
        \put(320,30){$2211\ldots$}
        \put(480,30){$222\ldots$}
        \end{picture}

$$\text{или}$$

        \setlength{\unitlength}{0.20mm}
        \begin{picture}(400,200)
        \put(280,160){$2222$}
        \put(300,150){\vector(2,-1){200}}
        \put(300,150){\vector(1,-2){50}}
        \put(300,150){\vector(-1,-2){50}}
        \put(300,150){\vector(-2,-1){200}}
        \put(80,30){$1222$}
        \put(220,30){$2122$}
        \put(330,30){$2212$}
        \put(490,30){$2221$}
        \end{picture}

        Становится ясно, что если $q$ содержит хотя бы одну единицу, то любой родитель $q2y$ также заканчивается на $2y$. А значит, $q$ не содержит  ни одной единицы, то есть $\exists e\in\mathbb{N}_0:$ $q=2^e$.  Таким образом, между  $t'(|c(t')|+1)$ и $t'(|c(t')|)=c(t')$ шаг ``вниз'' выглядит следующим образом: 
        $$2^e2y\to 2^e1y.$$
        
        \item $u:=e+1$. Ясно, что $|t'(|c(t')|+1)|=|c(t')|+1\ge n+1>n$, то есть $|2^e2y| > n\Longleftrightarrow |2^{u}y|>n \Longleftrightarrow |2^u|>n-|y|;$
        
        Также по Обозначению \ref{w} $|y|=h'(c(t'),w_m)=h'(t'(|c(t')|),w_m) \le h'(t'(n),w_m)=h'(t(n),w_m)=(\text{так как $|w_m|\ge m\ge n=|t(n)|$})=h'(t(n),w)<\left(\text{так как }t(n)\in \overline{P}(w,n,\delta)\right)<(1-\delta)n$.
        
        Таким образом, можно сделать вывод, что $2u=|2^{u}|>n-|y|>n-(1-\delta)n\ge{\delta n}\Longrightarrow 2u>{\delta n}$.
        
        То есть $u>\frac{\delta n}{2}$.

        \item И снова по Обозначению \ref{w} если $r\in\overline{|c(t')|+1,|w_m|}$, то $h'(t'(r),w_m)> h'(t'(|c(t')|),w_m)=h'(c(t'),w_m)=|y|$. А значит, номера вершин из множества $\left\{t'(r):\;r\in\overline{|c(t')|+1,|w_m|}\right\}$ заканчиваются на $2y$.
        \item Путь $t'$ выглядит следующим образом:
        $$t'(|w_m|)=w_m=x2y \to t'(|w_m|-1)\to \ldots \to $$
        $$\to t'(|c(t')|+1)=2^e2y=2^uy \to t'(|c(t')|)=c(t')=2^e1y \to$$
        $$\to \ldots \to t'(n)=t(n) .$$
    \end{itemize}
    
    \item $\exists x,y,z \in\mathbb{YF}:$ $w_m=x1y$, $c(t')=z2y$. 
    
    Как в данном случае проходит путь $t'$?
    
    \begin{itemize}
        \item По Обозначению \ref{w} $h'(t'(|c(t')|+1),w_m)>h'(t'(|c(t')|),w_m)=h'(c(t'),w_m)=|y|$. А значит $\exists q\in\mathbb{YF}:$ $t'(|c(t')|+1)=q1y$. Таким образом, между  $t'(|c(t')|+1)$ и $t'(|c(t')|)=c(t')$ шаг ``вниз'' выглядит следующим образом: 
        $$q1y\to z2y.$$
        
        Вспомним, как выглядят родители вершины графа Юнга--Фибоначчи:
        
        \setlength{\unitlength}{0.20mm}
        \begin{picture}(400,200)
        \put(270,160){$2221\ldots$}
        \put(300,150){\vector(2,-1){200}}
        \put(300,150){\vector(1,-2){50}}
        \put(300,150){\vector(-1,-2){50}}
        \put(300,150){\vector(-2,-1){200}}
        \put(60,30){$1221\ldots$}
        \put(220,30){$2121\ldots$}
        \put(320,30){$2211\ldots$}
        \put(480,30){$222\ldots$}
        \end{picture}

$$\text{или}$$

        \setlength{\unitlength}{0.20mm}
        \begin{picture}(400,200)
        \put(280,160){$2222$}
        \put(300,150){\vector(2,-1){200}}
        \put(300,150){\vector(1,-2){50}}
        \put(300,150){\vector(-1,-2){50}}
        \put(300,150){\vector(-2,-1){200}}
        \put(80,30){$1222$}
        \put(220,30){$2122$}
        \put(330,30){$2212$}
        \put(490,30){$2221$}
        \end{picture}

        Становится ясно, что если $q$ содержит хотя бы одну единицу, то любой родитель $q1y$ также заканчивается на $1y$. А значит, $q$ не содержит  ни одной единицы, то есть $\exists e\in\mathbb{N}_0: q=2^e$. Таким образом, между  $t'(|c(t')|+1)$ и $t'(|c\left(t'\right)|)=c(t')$ шаг ``вниз'' выглядит следующим образом: 
        $$2^e1y\to 2^ey.$$

        \item $u:=e$. Ясно, что $|c(t')|\ge n$, то есть $|2^ey| \ge n\Longleftrightarrow |2^{u}y|\ge n \Longleftrightarrow |2^u|\ge n-|y|;$
        
        Также по Обозначению \ref{w} $|y|=h'(c(t'),w_m)=h'(t'(|c(t')|),w_m) \le h'(t'(n),w_m)=h'(t(n),w_m)=(\text{так как $|w_m|\ge m\ge n=|t(n)|$})=h'(t(n),w)<\left(\text{так как }t(n)\in \overline{P}(w,n,\delta)\right)<(1-\delta)n$.
        
        Таким образом, можно сделать вывод, что $2u=|2^{u}|\ge n-|y| >n-(1-\delta)n\ge{\delta n}\Longrightarrow 2u>{\delta n}$.
        
        То есть $u>\frac{\delta n}{2}$.
        
        \item И снова по Обозначению \ref{w} если $r\in\overline{|c(t')|+1,|w_m|}$, то $h'(t'(r),w_m)> h'(t'(|c(t')|),w_m)=h'(c(t'),w_m)=|y|$. А значит, номера вершин из множества $\left\{t'(r):r\in\overline{|c(t')|+1,|w_m|}\right\}$ заканчиваются на $1y$.

        \item Путь $t'$ выглядит следующим образом:
        $$t'(|w_m|)=w_m=x1y \to t'(|w_m|-1)\to \ldots \to $$
        $$\to t'(|c(t')|+1)=2^e1y \to t'(|c(t')|)=c(t')=2^ey=2^uy \to$$
        $$\to \ldots \to t'(n)=t(n) .$$
    
    \end{itemize}
    \item $\exists x,y \in\mathbb{YF}:$ $w_m=xy$, $c(t')=y$. 
    
    По Обозначению \ref{w} $n\le |c(t')| =h'(c(t'),w_m)=h'(t'(|c(t')|),w_m) \le h'(t'(n),w_m)=h'(t(n),w_m)=
    (\text{так как $|w_m|\ge m\ge n=|t(n)|$})=h'(t(n),w)<
    \left(\text{так как }t(n)\in \overline{P}(w,n,\delta)\right)<(1-\delta)n \Longrightarrow n< (1-\delta)n$. Противоречие. Данный случай невозможен.
\end{enumerate}

Ясно, что все случаи разобраны.

Итак, в обоих возможных случаях можно сделать следующие выводы:

$\exists x,y \in\mathbb{YF},$ $u\in\mathbb{N}_0:$
\begin{itemize}
    \item Путь $t'$ выглядит следующим образом:
    $$w_m=xy \to \ldots \to 2^uy \to \ldots \to t'(n)=t(n); $$
    
    \item Номера всех вершин данного пути от $w_m=xy$ до $2^uy$ заканчиваются на $y;$
    
    \item $|y|<(1-\delta)n;$
    
    \item $u>\frac{\delta n}{2}$.
    
\end{itemize}

\begin{Oboz}
Пусть $w\in\mathbb{YF}_\infty$, $\delta\in(0,1)$, $n,m\in\mathbb{N}_0: m\ge n$. Тогда обозначим за
$$T(w,n,\delta,m)$$
множество таких путей $t\in T(\varepsilon,w_m)$, что $\exists x,y\in\mathbb{YF},$ $u\in\mathbb{N},$ такие что
\begin{itemize}
    \item Путь $t$ выглядит следующим образом:
    $$w_m=xy \to \ldots \to 2^uy \to \ldots \to \varepsilon ;$$
    
    \item Номера всех вершин данного пути от $w_m=xy$ до $2^uy$ заканчиваются на $y;$
    
    \item $|y|<(1-\delta)n;$
    
    \item $u>\frac{\delta n}{2}$.
\end{itemize}

\end{Oboz}

Мы рассмотрели путь $t\in T(\varepsilon,w_m)$ такой, что $t(n)\in \overline{P}(w,n,\delta)$, после чего поняли, что $t\in T(w,n,\delta,m)$.

Таким образом, ясно, что
$${{|\{t\in T(\varepsilon,w_m):\;t(n)\in \overline{P}(w,n,\delta) \}|}}\le |T(w,n,\delta,m)|\Longleftrightarrow$$
$$\Longleftrightarrow{\sum_{v\in\overline{P}(w,n,\delta)}{|\left\{t\in T(\varepsilon,w_m):\;t(n)=v\right\}|}}\le |T(w,n,\delta,m)|.$$

Теперь ясно, что если мы просуммируем по всем тройкам $(x,y,u)\in\mathbb{YF}\times\mathbb{YF}\times\mathbb{N},$ таким что
\begin{itemize}
    \item $w_m=xy,$ 
    
    \item $|y|<(1-\delta)n,$
    
    \item $u>\frac{\delta n}{2},$
\end{itemize}
количество путей вида
$$w_m=xy \to \ldots \to 2^{u}y \to \ldots \to \varepsilon ,$$
таких что номера всех вершин данного пути от $w_m$ до $2^uy$ заканчиваются на $y$, 
то мы посчитаем все пути из $T(w,n,\delta,m)$ хотя бы один раз.

А теперь вспомним определение $\overline{P}(w,n,\delta,m)$, поймём, что при фиксированных $x,y$ и $u$ количество путей вида
$$w_m=xy \to \ldots \to 2^{u}y \to \ldots \to \varepsilon ,$$
таких что номера всех вершин данного пути от $w_m$ до $2^uy$ заканчиваются на $y,$ равно $${{d(\varepsilon,2^uy)d(2^u,x)}}$$
и сделаем вывод, что
$$|T(w,n,\delta,m)| \le  \sum_{(x,y,u)\in\overline{P}(w,n,\delta,m)}\left({d(\varepsilon,2^uy)d(2^u,x)}\right).$$
    
Таким образом, мы доказали, что
$${\sum_{v\in\overline{P}(w,n,\delta)}{|\{t\in T(\varepsilon,w_m):\;t(n)=v\}|}} \le |T(w,n,\delta,m)| \le \sum_{(x,y,u)\in\overline{P}(w,n,\delta,m)}\left({{d(\varepsilon,2^uy)d(2^u,x)}}\right)\Longrightarrow$$
$$\Longrightarrow{\sum_{v\in\overline{P}(w,n,\delta)}{|\{t\in T(\varepsilon,w_m):\;t(n)=v\}|}} \le  \sum_{(a,b,k)\in\overline{P}(w,n,\delta,m)}\left({{d(\varepsilon,2^kb)d(2^k,a)}}\right)\Longleftrightarrow$$
$$\Longleftrightarrow{\sum_{v\in\overline{P}(w,n,\delta)}\frac{d(\varepsilon,v)d(v,w_m)}{d(\varepsilon,w_m)}}\le \sum_{(a,b,k)\in\overline{P}(w,n,\delta,m)}{\frac{d(\varepsilon,2^kb)d(2^k,a)}{d(\varepsilon,w_m)}},$$
что и требовалось.

Лемма доказана.
\end{proof}

\begin{Oboz} Пусть $x\in\mathbb{YF},$ $i\in\mathbb{N}_0:$ $i\in \overline{1,d(x)}$. Тогда
$$g'(x,i):=g(x,i)-2i+2.$$
\end{Oboz}

    \begin{Oboz}
    Пусть $a\in\mathbb{YF}$, $k\in\mathbb{N}_0$, $i=(i_1,\ldots,i_{d(a)-k})\in\mathbb{N}^{d(a)-k}:$ $d(a)\ge k$, $1\le i_1<i_2<\ldots<i_{d(a)-k}\le d(a)$. 
    
    Тогда:
    \begin{itemize}
        \item     
        за $d(2^k,a,i)$ обозначим число таких путей ``вниз'' из $a$ в $2^k$, при прохождении которых удаляются те и только те двойки, которые изначально были $i_1$-ой, $i_2$-ой, $\ldots$ и $i_{d(a)-k}$-ой в номере $a$, если считать справа;
    
        \item  за $a(i)$ обозначим вершину графа Юнга--Фибоначчи, $(a(i)\in\mathbb{YF})$, номер которой получается из номера вершины $a$ удалением $i_1$-ой, $i_2$-ой, $\ldots$ и $i_{d(a)-k}$-ой двоек, если считать справа.
   
    \end{itemize}
\end{Oboz}
    
    \begin{Prop}    \label{2power}
    Пусть $a\in\mathbb{YF}$, $k\in\mathbb{N}_0$, $i=(i_1,\ldots,i_{d(a)-k})\in\mathbb{N}^{d(a)-k}:$ $d(a)\ge k$, $1\le i_1<i_2<\ldots<i_{d(a)-k}\le d(a)$. Тогда
    $$d(2^{k},a,i)=\prod_{j=1}^{d(a)-k} (g'(a,i_j)+2j-2).$$
    \end{Prop}
    \begin{proof}

    Хотим доказать, что 
    $$d(2^{k},a,i)=d(\varepsilon,a(i))=\prod_{j=1}^{d(a)-k} (g'(a,i_j)+2j-2).$$
    
    Для начала поймём, что 
    $$d(2^{k},a,i)=d(\varepsilon,a(i)):$$
    Это несложно понять по следующим двум картинкам. Вверху нарисовано, как устроен шаг пути ``вниз'' из вершины $a$ в $2^k$. Перечёркнуты тут двойки, которые в процессе прохождения пути не должны быть удалены. Внизу нарисовано, как устроен соответствующий шаг пути ``вниз'' из вершины $a(i)$ в $\varepsilon$. Ясно, как из этого следует биекция между этими множествами путей.
    
        \setlength{\unitlength}{0.20mm}
        \begin{picture}(400,200)
        \put(250,160){$a=2\xcancel{2}2\xcancel{2}21...$}
        \put(300,150){\vector(2,-1){200}}
        \put(300,150){\vector(1,-2){50}}
        \put(300,150){\vector(-1,-2){50}}
        \put(300,150){\vector(-2,-1){200}}
        \put(60,30){$1\xcancel{2}2\xcancel{2}21...$}
        \put(220,30){$2\xcancel{2}1\xcancel{2}21...$}
        \put(320,30){$2\xcancel{2}2\xcancel{2}11...$}
        \put(480,30){$2\xcancel{2}2\xcancel{2}2...$}
        \end{picture}

        \setlength{\unitlength}{0.20mm}
        \begin{picture}(400,200)
        \put(250,160){$a(i)=2221\ldots$}
        \put(300,150){\vector(2,-1){200}}
        \put(300,150){\vector(1,-2){50}}
        \put(300,150){\vector(-1,-2){50}}
        \put(300,150){\vector(-2,-1){200}}
        \put(60,30){$1221\ldots$}
        \put(220,30){$2121\ldots$}
        \put(320,30){$2211\ldots$}
        \put(480,30){$222\ldots$}
        \end{picture}
    
    Иными словами, ясно, что двойки, которые нельзя удалить, ни на что не влияют, так как их нельзя удалить, а также, так как их расположение относительно других символов никак не влияет на то, какие операции можно сделать с остальными символами.

    Теперь хотим понять, что
    $$d(\varepsilon,a(i))=\prod_{j=1}^{d(a)-k} (g'(a,i_j)+2j-2).$$

    По определению функции $g'$ и формуле для числа путей ``вниз'' в $\varepsilon$ можно понять, что
    $$\prod_{j=1}^{d(a)-k} (g'(a,i_j)+2j-2)=\prod_{j=1}^{d(a)-k} (g(a,i_j)-2i_j+2+2j-2)=\prod_{j=1}^{d(a)-k} (g(a,i_j)-2i_j+2j)=$$
    \begin{center}
    =(Так как ясно, что $2i_j-2j$ -- это удвоенное количество таких двоек в номере $a$, которые нельзя удалять, и которые при этом стоят справа от $j$-ой ``разрешённой'' двойки)=
    \end{center}
    $$=\prod_{j=1}^{d(a)-k} g(a(i),j)=\prod_{j=1}^{d(a(i))} g(a(i),j)=d(\varepsilon,a(i)).$$
    
    Таким образом, мы доказали, что 
    $$d(2^{k},a,i)=d(\varepsilon,a(i))=\prod_{j=1}^{d(a)-k} (g'(a,i_j)+2j-2),$$
    что и требовалось.

    Утверждение доказано.
    \end{proof}

\begin{Prop} \label{car}
Пусть $a\in \mathbb{YF}$, $k\in\mathbb{N}_0$. Тогда
$$d(2^k,a)=\sum_{1\le i_1<i_2<...<i_{d(a)-k}\le d(a)}\left(\prod_{j=1}^{d(a)-k} (g'(a,i_j)+2j-2)\right).$$
\end{Prop}
\begin{proof}
Рассмотрим два случая:
\begin{enumerate}
    \item $d(a)<k$. 
    
    Ясно, что при каждом шаге ``вниз'' количество двоек не увеличивается, так как мы можем только заменить двойку на единицу или удалить единицу

        \setlength{\unitlength}{0.20mm}
        \begin{picture}(400,200)
        \put(270,160){$2221\ldots$}
        \put(300,150){\vector(2,-1){200}}
        \put(300,150){\vector(1,-2){50}}
        \put(300,150){\vector(-1,-2){50}}
        \put(300,150){\vector(-2,-1){200}}
        \put(60,30){$1221\ldots$}
        \put(220,30){$2121\ldots$}
        \put(320,30){$2211\ldots$}
        \put(480,30){$222\ldots$}
        \end{picture}

$$\text{или}$$

        \setlength{\unitlength}{0.20mm}
        \begin{picture}(400,200)
        \put(280,160){$2222$}
        \put(300,150){\vector(2,-1){200}}
        \put(300,150){\vector(1,-2){50}}
        \put(300,150){\vector(-1,-2){50}}
        \put(300,150){\vector(-2,-1){200}}
        \put(80,30){$1222$}
        \put(220,30){$2122$}
        \put(330,30){$2212$}
        \put(490,30){$2221$}
        \end{picture}
        
    А значит, так как справа сумма пустая,
    $$d(2^k,a)=0=\sum_{1\le i_1<i_2<...<i_{d(a)-k}\le d(a)}\left(\prod_{j=1}^{d(a)-k} (g'(a,i_j)+2j-2)\right).$$

    \item $d(a)\ge k.$

    Ясно, что при любом пути ``вниз'' из вершины $a$ в вершину $2^k$ при $d(a)\ge k$ какие-либо $k$ из изначально находящихся $d(a)$ двоек в номере $a$ остаются нетронутыми, а $d(a)-k$ --- удаляются. Таким образом, для завершения доказательства Утверждения осталось только просуммировать Утверждение \ref{2power} по всем $i=(i_1,\ldots,i_{d(a)-k})\in\mathbb{N}^{d(a)-k}$: $1\le i_1<i_2<\ldots<i_{d(a)-k}\le d(a)$. 
    $$d(2^k,a)=\sum_{\begin{smallmatrix}i=(i_1,\ldots,i_{d(a)-k}):        \\1\le i_1<i_2<...<i_{d(a)-k}\le d(a)\end{smallmatrix}}d(2^{k},a,i)=\sum_{1\le i_1<i_2<...<i_{d(a)-k}\le d(a)}\left(\prod_{j=1}^{d(a)-k} (g'(a,i_j)+2j-2)\right),$$
    что и требовалось.

    Утверждение доказано.
\end{enumerate}

\end{proof}

\begin{Prop} \label{inf}
Пусть $w\in\mathbb{YF}_\infty^+$. Тогда 
$$\sum_{i=1}^{d(w)} \frac{1}{g(w,i)}<\infty.$$
\end{Prop}

\begin{proof}

$$w\in\mathbb{YF}_\infty^+\Longrightarrow \pi(w)>0\Longleftrightarrow\prod_{i:g(w,i)> 1  } \frac{g(w,i)-1}{g(w,i)}>0\Longleftrightarrow$$
$$\Longleftrightarrow \prod_{i:g(w,i)> 1  } \frac{g(w,i)}{g(w,i)-1}<\infty \Longleftrightarrow \prod_{i:g(w,i)> 1  } \left(1 + \frac{1}{g(w,i)-1}\right)<\infty.$$
$$\infty>\prod_{i:g(w,i)> 1  } \left(1 + \frac{1}{g(w,i)-1}\right)\ge \prod_{i:g(w,i)> 1  } \left(1 + \frac{1}{g(w,i)}\right)\ge 1+\sum_{i:g(w,i)> 1  }  \frac{1}{g(w,i)} \ge \sum_{i=1}^{d(w)} \frac{1}{g(w,i)}.$$
\end{proof}

\begin{Lemma} \label{rap}
Пусть $w\in \mathbb{YF}_\infty^+,$ $s\in(0,1)$. Тогда $\exists c_1,c_2\in\mathbb{N}_0:$ $\forall a,b\in\mathbb{YF},$ $m,k\in\mathbb{N}_0:$ $w_m=ab,$ $k\ge c_2$
$$\frac{\displaystyle d\left(2^k,a\right)}{\displaystyle d\left(\varepsilon,a1^{|b|}\right)} \le c_1\frac{s^{k-c_2}}{(k-c_2)!}.$$
\end{Lemma}
\begin{proof}

Давайте рассмотрим $a,b\in\mathbb{YF},$ $m,k\in\mathbb{N}_0:$ $w_m=ab,$ $d(a)\ge k$.

Очевидно, что в данном случае $|a|\ge 2k$. 

Ясно, что $d\left(2^k,a\right)\le d\left(2^k,a1^{|b|}\right)$, так как каждому пути $t\in T\left(2^k,a\right)$ ``вниз'' $$a=t(|a|)\to t(|a|-1)\to t(|a|-2)\to \ldots\to  t(|2k|)=2^k$$
можно однозначно сопоставить путь $t'\in T\left(2^k,a1^{|b|}\right)$ ``вниз'' $$a1^{|b|}=t(|a|)1^{|b|}\to t(|a|-1)1^{|b|}\to t(|a|-2)1^{|b|}\to \ldots\to$$
$$\to \ldots\to  t(|2k|)1^{|b|}=2^k1^{|b|}\to 2^k1^{|b|-1}\to 2^k1^{|b|-2}\to \ldots \to 2^k,$$ причём любой паре таких путей соответствуют разные пути.

Давайте посчитаем:
$$\frac{\displaystyle d\left(2^k,a\right)}{\displaystyle d\left(\varepsilon,a1^{|b|}\right)} \le \frac{d\left(2^k,a1^{|b|}\right)}{d\left(\varepsilon,a1^{|b|}\right)}=\left(\text{По Утверждению \ref{car} при $\left(a1^{|b|}\right)\in \mathbb{YF}$, $k\in\mathbb{N}_0$}\right)=$$
$$=\frac{\displaystyle\sum_{1\le i_1<i_2<...<i_{d(a)-k}\le d(a)}\left(\prod_{j=1}^{d(a)-k} \left(g'\left(a1^{|b|},i_j\right)+2j-2\right)\right)}{\displaystyle d\left(\varepsilon,a1^{|b|}\right)}=$$
$$=\left(\text{По формуле для количества путей ``вниз'' в $\varepsilon$}\right)=$$
$$=\frac{\displaystyle\sum_{1\le i_1<i_2<...<i_{d(a)-k}\le d(a)}\left(\prod_{j=1}^{d(a)-k} \left(g'\left(a1^{|b|},i_j\right)+2j-2\right)\right)}{\displaystyle\prod_{j=1}^{d\left(a1^{|b|}\right)} g\left(a1^{|b|},j\right)}=$$
$$=\frac{\displaystyle\sum_{1\le i_1<i_2<...<i_{d(a)-k}\le d(a)}\left(\prod_{j=1}^{d(a)-k} \left(g'\left(a1^{|b|},i_j\right)+2j-2\right)\right)}{\displaystyle\prod_{j=1}^{d(a)} g\left(a1^{|b|},j\right)}=$$
$$=\text{(По определению $g'$)}=$$
$$=\frac{\displaystyle\sum_{1\le i_1<i_2<...<i_{d(a)-k}\le d(a)}\left(\prod_{j=1}^{d(a)-k} \left(g\left(a1^{|b|},i_j\right)-2i_j+2+2j-2\right)\right)}{\displaystyle\prod_{j=1}^{d(a)} g\left(a1^{|b|},j\right)}=$$
$$=\frac{\displaystyle\sum_{1\le i_1<i_2<...<i_{d(a)-k}\le d(a)}\left(\prod_{j=1}^{d(a)-k} \left(g\left(a1^{|b|},i_j\right)-2i_j+2j\right)\right)}{\displaystyle\prod_{j=1}^{d(a)} g\left(a1^{|b|},j\right)}\le$$
$$\le \frac{\displaystyle\sum_{1\le i_1<i_2<...<i_{d(a)-k}\le d(a)}\left(\prod_{j=1}^{d(a)-k} g\left(a1^{|b|},i_j\right)\right)}{\displaystyle\prod_{j=1}^{d(a)} g\left(a1^{|b|},j\right)}= {\sum_{1\le i_1<i_2<...<i_{k}\le d(a)}\left( \frac{1}{\displaystyle\prod_{j=1}^{k} g\left(a1^{|b|},i_j\right)} \right)} \le$$
$$\le \left(\text{Из определения функции $g$ ясно, что $\left\{g\left(a1^{|b|},j\right)\right\}_{j=1}^{d\left(a1^{|b|}\right)}=\left\{g\left(a1^{|b|},j\right)\right\}_{j=1}^{d(a)}\subseteq \{g(ab,j)\}_{j=1}^{d(ab)}$}\right) \le $$
$$\le {\sum_{1\le i_1<i_2<...<i_{k}\le d(ab)}\left( \frac{1}{\displaystyle\prod_{j=1}^{k} g(ab,i_j)} \right)} ={\sum_{1\le i_1<i_2<...<i_{k}\le d(w_m)}\left(\frac{1}{\displaystyle \prod_{j=1}^{k} g(w_m,i_j)}\right)}\le$$
$$\le \left(\text{Из определения функции $g$ ясно, что $\{g(w_m,j)\}_{j=1}^{d(w_m)}\subseteq \{g(w,j)\}_{j=1}^{d(w)}$}\right) \le $$
$$\le {\sum_{1\le i_1<i_2<...<i_{k}\le d(w)}\left(\frac{1}{\displaystyle \prod_{j=1}^{k} g(w,i_j)}\right)} .$$
Иными словами, последние два неравенства верны, так как $w$ содержит все двойки из $a1^{|b|}$, причём именно на тех же местах.

Заметим, что данное выражение не зависит от $a,b$ и $m$.

По Утверждению \ref{inf} при нашем $w\in \mathbb{YF}_\infty^+$
$$\sum_{i=1}^{d(w)} \frac{1}{g(w,i)}<\infty \Longrightarrow \exists S\in\mathbb{N}_0: \; \sum_{i=S+1}^{d(w)} \frac{1}{g(w,i)}<s.$$

Выберем такое $S\in\mathbb{N}_0$. Пусть $c_2=S$. 

Несложно заметить (из раскрытия скобок при возведении в степень), что при $k\ge c_2=S$
$${\sum_{S+1\le i_1<i_2<...<i_{k}\le d(w)}\left( \frac{\displaystyle 1}{\displaystyle\prod_{j=1}^{k} g(w,i_j)}\right) }\le \frac{\displaystyle\left(\sum_{i=S+1}^{d(w)} \frac{1}{g(w,i)}\right)^k}{k!} \le\frac{s^k}{k!}; $$
$${\sum_{1\le i_1<S+1\le i_2<...<i_{k}\le d(w)}\left(\frac{1}{\displaystyle\prod_{j=1}^{k} g(w,i_j)}\right)}\le \left(\sum_{i=1}^S\frac{1}{g(w,i)}\right) \frac{\displaystyle\left(\sum_{i=S+1}^{d(w)} \frac{1}{g(w,i)}\right)^{k-1}}{(k-1)!}\le$$
$$\le  \left(\sum_{i=1}^S\frac{1}{g(w,i)}\right) \frac{s^{k-1}}{(k-1)!};$$
$${\sum_{1\le i_1< i_2<S+1\le i_3<...<i_{k}\le d(w)}\left(\frac{1}{\displaystyle\prod_{j=1}^{k} g(w,i_j)}\right)}\le \left(\sum_{i=1}^S\frac{1}{g(w,i)}\right)^2 \frac{\displaystyle\left(\sum_{i=S+1}^{d(w)} \frac{1}{g(w,i)}\right)^{k-2}}{(k-2)!}\le$$
$$\le  \left(\sum_{i=1}^S\frac{1}{g(w,i)}\right)^2\frac{s^{k-2}}{(k-2)!}; $$

$$\ldots$$

$${\sum_{1\le i_1< i_2<\ldots<i_l<S+1\le i_{l+1}<\ldots<i_{k}\le d(w)}\left(\frac{1}{\displaystyle\prod_{j=1}^{k} g(w,i_j)}\right)}\le \left(\sum_{i=1}^S\frac{1}{g(w,i)}\right)^l \frac{\displaystyle\left(\sum_{i=S+1}^{d(w)} \frac{1}{g(w,i)}\right)^{k-l}}{(k-l)!}\le$$
$$\le  \left(\sum_{i=1}^S\frac{1}{g(w,i)}\right)^l \frac{s^{k-l}}{(k-l)!};$$

$$\ldots$$

$${\sum_{1\le i_1< i_2<\ldots<i_S<S+1\le i_{S+1}<\ldots<i_{k}\le d(w)}\left(\frac{1}{\displaystyle\prod_{j=1}^{k} g(w,i_j)}\right)}\le \left(\sum_{i=1}^S\frac{1}{g(w,i)}\right)^S \frac{\displaystyle\left(\sum_{i=S+1}^{d(w)} \frac{1}{g(w,i)}\right)^{k-S}}{(k-S)!}\le$$
$$ \le \left(\sum_{i=1}^S\frac{1}{g(w,i)}\right)^S\frac{s^{k-S}}{(k-S)!} .$$

Выше написано $S+1$ неравенство, в $l$-ом неравенстве (при нумерации с нуля) слева написана сумма по таким $k$ $i$-шкам, что ровно $l$ из них строго меньше, чем $S+1$, а остальные $k-l$ -- не меньше. И данную сумму мы оцениваем сверху. Ясно, что других вариантов расположения $k$ $i$-шек нет, а значит, просуммировав данные неравенства, получим верхнюю оценку для всех таких наборов из $k$ $i$-шек:

$${\sum_{1\le i_1<i_2<...<i_{k}\le d(w)}\frac{1}{\displaystyle\prod_{j=1}^{k} g(w,i_j)}} \le$$
$$\le \frac{s^k}{k!}+\left(\sum_{i=1}^S\frac{1}{g(w,i)}\right)\frac{s^{k-1}}{(k-1)!} +\left(\sum_{i=1}^S\frac{1}{g(w,i)}\right)^2\frac{s^{k-2}}{(k-2)!}+ \ldots +\left(\sum_{i=1}^S\frac{1}{g(w,i)}\right)^S\frac{s^{k-S}}{(k-S)!}\le$$
$$\le \frac{s^k}{(k-S)!}+\left(\sum_{i=1}^S\frac{1}{g(w,i)}\right)\frac{s^{k-1}}{(k-S)!} +\left(\sum_{i=1}^S\frac{1}{g(w,i)}\right)^2\frac{s^{k-2}}{(k-S)!}+ \ldots +\left(\sum_{i=1}^S\frac{1}{g(w,i)}\right)^S\frac{s^{k-S}}{(k-S)!}\le$$
$$\le\text{(Так как $s\in (0,1)$)}\le$$
$$\le \frac{s^{k-S}}{(k-S)!}+\left(\sum_{i=1}^S\frac{1}{g(w,i)}\right)\frac{s^{k-S}}{(k-S)!} +\left(\sum_{i=1}^S\frac{1}{g(w,i)}\right)^2\frac{s^{k-S}}{(k-S)!}+ \ldots +\left(\sum_{i=1}^S\frac{1}{g(w,i)}\right)^S\frac{s^{k-S}}{(k-S)!}=$$
$$= \frac{s^{k-S}}{(k-S)!}\left(1+\left(\sum_{i=1}^S\frac{1}{g(w,i)}\right)+\left(\sum_{i=1}^S\frac{1}{g(w,i)}\right)^2+\ldots+\left(\sum_{i=1}^S\frac{1}{g(w,i)}\right)^S \right)=$$
$$=\frac{s^{k-S}}{(k-S)!}\left(\sum_{j=0}^S\left(\sum_{i=1}^S\frac{1}{g(w,i)}\right)^j\right)\le\frac{s^{k-S}}{(k-S)!}\left\lceil \sum_{j=0}^S\left(\sum_{i=1}^S\frac{1}{g(w,i)}\right)^j \right\rceil. $$

Пусть $$c_1=\left\lceil \sum_{j=0}^S\left(\sum_{i=1}^S\frac{1}{g(w,i)}\right)^j \right\rceil<\infty.$$ 

При выбранных нами $c_1,c_2\in\mathbb{N}_0$ и произвольных $a,b\in\mathbb{YF},$ $m,k\in\mathbb{N}_0:$ $w_m=ab,$ $k\ge c_2$ есть два случая:
\begin{enumerate}
\item $d(a)\ge k$.

В данном случае мы поняли, что $$\frac{\displaystyle d(2^k,a)}{\displaystyle d\left(\varepsilon,a1^{|b|}\right)}\le {\sum_{1\le i_1<i_2<...<i_{k}\le d(w)}\left(\frac{1}{\displaystyle \prod_{j=1}^{k} g(w,i_j)}\right)}\le\frac{s^{k-S}}{(k-S)!} \left\lceil \sum_{j=0}^S\left(\sum_{i=1}^S\frac{1}{g(w,i)}\right)^j \right\rceil =c_1\frac{s^{k-c_2}}{(k-c_2)!}.$$
\item $d(a)<k$.

Заметим, что при каждом шаге ``вниз'' количество двоек не увеличивается, так как мы можем только заменить двойку на единицу или удалить единицу, а значит, если $d(a)<k$, то $d(2^k,a)=0$. Таким образом, в данном случае
$$\frac{\displaystyle d\left(2^k,a\right)}{\displaystyle d\left(\varepsilon,a1^{|b|}\right)}=0 \le c_1\frac{s^{k-c_2}}{(k-c_2)!}.$$
\end{enumerate}

Мы поняли, что при выбранных нами $c_1,c_2\in\mathbb{N}_0$ и произвольных $a,b\in\mathbb{YF},$ $m,k\in\mathbb{N}_0:$ $w_m=ab,$ $k\ge c_2$ в любом случае
$$\frac{\displaystyle d\left(2^k,a\right)}{\displaystyle d\left(\varepsilon,a1^{|b|}\right)}\le c_1\frac{s^{k-c_2}}{(k-c_2)!},$$
что и требовалось.

Лемма доказана.

\end{proof}

\begin{Prop} \label{razbivaem}
Пусть $x,x',x''\in\mathbb{YF}:$ $x=x'x''$. Тогда
$$d(\varepsilon,x)=d(\varepsilon,x'')d\left(\varepsilon,x'1^{\left|x''\right|}\right).$$
\end{Prop}
\begin{proof} 
Воспользуемся формулой для числа путей ``вниз'' в $\varepsilon$ и определением функции $g$:
$$d(\varepsilon,x)=\prod_{i=1}^{d(x)}g(x,i)=\left(\prod_{i=1}^{d\left(x''\right)}g(x,i)\right)\left(\prod_{i=d\left(x''\right)+1}^{d(x)}g(x,i)\right)=d\left(\varepsilon,x''\right)d\left(\varepsilon,x'1^{\left|x''\right|}\right).$$

\end{proof}

Теперь давайте докажем Теорему. Самое время вспомнить формулировку:

\renewcommand{\labelenumi}{\arabic{enumi}$)$}
\renewcommand{\labelenumii}{\arabic{enumi}.\arabic{enumii}$^\circ$}
\renewcommand{\labelenumiii}{\arabic{enumi}.\arabic{enumii}.\arabic{enumiii}$^\circ$}

Пусть $w\in\mathbb{YF}_\infty^+,$ $\delta\in(0,1)$. Тогда
\begin{enumerate}
    \item $$\lim_{n \to \infty}{\sum_{v\in \overline{P}(w,n,\delta)}\mu_w(v)=0};$$
    \item $$\lim_{n \to \infty}{\sum_{v\in P(w,n,\delta)}\mu_w(v)=1}.$$
\end{enumerate}

Давайте при наших $w\in\mathbb{YF}_\infty^+$ и $\delta\in(0,1)$  зафиксируем произвольные $m,n\in\mathbb{N}_0: m\ge n\ge 1 $. 

Ясно, что если $(a,b,k)\in\overline{P}(w,n,\delta,m)$ $\left(\text{то есть при $a,b\in\mathbb{YF}$, $k\in\mathbb{N}$: $ab=w_m,$ $|b|< (1-\delta)n$, $ k>\frac{\delta n}{2}$}\right)$, то
$$\frac{|b|}{k}< \frac{(1-\delta)n}{\frac{\delta n}{2}}=\frac{2(1-\delta )n}{\delta n}=\frac{2(1-\delta)}{\delta}<\frac{2}{\delta}\Longrightarrow |b|<\frac{2}{\delta}k.$$
А значит, если $(a,b,k)\in\overline{P}(w,n,\delta,m)$, то
$$\frac{d\left(\varepsilon,2^kb\right)}{d\left(\varepsilon,b\right)}=\left(\text{По Утверждению \ref{razbivaem} при $(2^kb),2^k,b\in\mathbb{YF}$}\right)=\frac{d\left(\varepsilon,b\right)d\left(\varepsilon,2^k1^{|b|}\right)}{d\left(\varepsilon,b\right)}=$$
$$=d\left(\varepsilon,2^k1^{|b|}\right)=(\text{По формуле для числа путей ``вниз'' в $\varepsilon$})=$$
$$=\prod_{i=1}^{d\left(2^k1^{|b|}\right)}g\left(2^k1^{|b|},i\right)=\prod_{i=1}^{k}g\left(2^k1^{|b|},i\right)=(\text{По определению функции $g$})=$$
$$=\prod_{i=1}^{k}(|b|+2i-1)\le \prod_{i=1}^{k}(|b|+2k-1)= (|b|+2k-1)^k< \left(\frac{2}{\delta} k+2k-1\right)^k.$$

Таким образом, мы поняли, что при $(a,b,k)\in\overline{P}(w,n,\delta,m)$
$$\frac{d\left(\varepsilon,2^kb\right)}{d\left(\varepsilon,b\right)}< \left(\frac{2}{\delta} k+2k-1\right)^k.$$

Применим Лемму \ref{rap} при нашем $w\in\mathbb{YF}_\infty^+$ и $\displaystyle s=\frac{1}{2e\left(\frac{2}{\delta}+2\right)}$ (ясно, что $s\in(0,1)$). По этой Лемме $\exists c_1,c_2\in\mathbb{N}_0$: $\forall a,b\in\mathbb{YF},$ $m,k\in\mathbb{N}_0$: $w_m=ab,$ $k\ge c_2$
$$\frac{d(2^k,a)}{d(\varepsilon, a1^{|b|})} \le c_1 \frac{\displaystyle\left(\frac{1}{2e(\frac{2}{\delta}+2)}\right)^{k-c_2}}{(k-c_2)!}.$$

Рассматриваем данные $c_1$ и $c_2$.

Перемножив два неравенства и вспомнив определение $\overline{P}(w,n,\delta,m)$, получаем, что $\forall a,b\in\mathbb{YF},$ $k\in\mathbb{N}$: $(a,b,k)\in\overline{P}(w,n,\delta,m),$ $k\ge c_2$
$${\frac{d\left(\varepsilon,2^kb\right)d\left(2^k,a\right)}{d\left(\varepsilon,b\right)d\left(\varepsilon,a1^{|b|}\right)}}< \left(\frac{2}{\delta} k+2k-1\right)^kc_1 \frac{\displaystyle\left(\frac{1}{2e(\frac{2}{\delta}+2)}\right)^{k-c_2}}{(k-c_2)!}\le \left(\frac{2}{\delta} k+2k\right)^kc_1 \frac{\displaystyle\left(\frac{1}{2e(\frac{2}{\delta}+2)}\right)^{k-c_2}}{(k-c_2)!}\le$$
$$ \le\left(k\left(\frac{2}{\delta} +2\right)\right)^kc_1 \frac{\displaystyle\left(\frac{1}{2e\left(\frac{2}{\delta}+2\right)}\right)^{k-c_2}}{(k-c_2)!} = k^k\left(\frac{2}{\delta} +2\right)^{k} c_1 \frac{\displaystyle\left(\frac{1}{2e}\right)^{k-c_2}\left(\frac{2}{\delta} +2\right)^{c_2-k}}{(k-c_2)!} =$$
$$=k^k\left(\frac{2}{\delta} +2\right)^{c_2} c_1 \frac{\displaystyle\left(\frac{1}{2e}\right)^{k-c_2}}{(k-c_2)!}=c_1\left(\frac{2}{\delta} +2\right)^{c_2} k^k \frac{\displaystyle\left(\frac{1}{2e}\right)^{k-c_2}}{(k-c_2)!}=\left(\text{при $C_1=c_1\left(\frac{2}{\delta}+2\right)^{c_2}$}\right)= C_1  k^k \frac{\displaystyle\left(\frac{1}{2e}\right)^{k-c_2}}{(k-c_2)!}.$$

Пусть $k>c_2$. Тогда ясно, что $(k-c_2)\in\mathbb{N}$. Применим формулу Стирлинга для $(k-c_2)$. Она гласит, что $\exists \theta_{k-c_2}\in(0,1)$: $$(k-c_2)!=\sqrt{2\pi(k-c_2)}  \left(\frac{k-c_2}{e}\right)^{k-c_2}\exp\frac{\theta_{k-c_2}}{12(k-c_2)}.$$

Таким образом, мы получаем, что $\forall a,b\in\mathbb{YF},$ $k\in\mathbb{N}$: $(a,b,k)\in\overline{P}(w,n,\delta,m)$, $k>c_2$
$${\frac{d\left(\varepsilon,2^kb\right)d\left(2^k,a\right)}{d(\varepsilon,b)d\left(\varepsilon,a1^{|b|}\right)}}<\left(\text{при $C_1=c_1\left(\frac{2}{\delta}+2\right)^{c_2}$}\right)< C_1  k^k \frac{\displaystyle\left(\frac{1}{2e}\right)^{k-c_2}}{(k-c_2)!}=$$
$$=C_1 k^k \frac{\displaystyle\left(\frac{1}{2e}\right)^{k-c_2}}{\displaystyle\sqrt{2\pi(k-c_2)}  \left(\frac{k-c_2}{e}\right)^{k-c_2}\exp\frac{\theta_{k-c_2}}{12(k-c_2)}}\le C_1 k^k \frac{\displaystyle\left(\frac{1}{2e}\right)^{k-c_2}}{\displaystyle\sqrt{2\pi(k-c_2)}  \left(\frac{k-c_2}{e}\right)^{k-c_2}}\le$$
$$\le C_1 k^k \frac{\left(\frac{1}{2}\right)^{k-c_2}}{\sqrt{2\pi(k-c_2)}  \left({k-c_2}\right)^{k-c_2}}\le C_1 k^k \frac{\left(\frac{1}{2}\right)^{k-c_2}}{  \left({k-c_2}\right)^{k-c_2}}= C_1 k^k \frac{\left(\frac{1}{2}\right)^{k}}{  \left({k-c_2}\right)^{k}}\cdot\frac{  \left({k-c_2}\right)^{c_2}}{\left(\frac{1}{2}\right)^{c_2}}\le$$
$$\le C_1k^k  \frac{\left(\frac{1}{2}\right)^{k}}{  \left({k-c_2}\right)^{k}}k^{c_2}2^{c_2}=2^{c_2}C_1k^{c_2}k^k  \frac{\left(\frac{1}{2}\right)^{k}}{  \left({k-c_2}\right)^{k}}=\left(\text{при $C_2=2^{c_2}C_1$}\right) = $$
$$=C_2 k^{c_2} k^k \frac{\left(\frac{1}{2}\right)^{k}}{  \left({k-c_2}\right)^{k}} = C_2k^{c_2} \left(\frac{k}{k-c_2}\right)^k {\left(\frac{1}{2}\right)^{k}}.$$

Как известно,
$$\lim_{k\to\infty}\left(\frac{k}{k-c_2}\right)^k=e^{c_2},$$
а значит, существует $K\in\mathbb{N}$: $K>c_2$, $\forall k\in\mathbb{N}:$ $k\ge K$
$$\left(\frac{k}{k-c_2}\right)^k<2e^{c_2}.$$
Зафиксируем данное $K$.

Мы получаем, что $\forall a,b\in\mathbb{YF},$ $k\in\mathbb{N}$: $(a,b,k)\in\overline{P}(w,n,\delta,m)$, $k\ge K$
$${\frac{d\left(\varepsilon,2^kb\right)d\left(2^k,a\right)}{d(\varepsilon,b)d\left(\varepsilon,a1^{|b|}\right)}}<\left(\text{при $C_2=2^{c_2}C_1$, $C_1=c_1\left(\frac{2}{\delta}+2\right)^{c_2}$}\right)< C_2k^{c_2} \left(\frac{k}{k-c_2}\right)^k {\left(\frac{1}{2}\right)^{k}}
<$$
$$<  C_2k^{c_2} 2e^{c_2} {\left(\frac{1}{2}\right)^{k}}=2C_2e^{c_2}k^{c_2}{\left(\frac{1}{2}\right)^{k}}=\left(\text{при $C=2C_2e^{c_2}$}\right)=Ck^{c_2} {\left(\frac{1}{2}\right)^{k}}.$$

Ясно, что экспонента растёт быстрее многочлена, а значит, $\exists K'\in\mathbb{N}:$ $K'\ge K,$ $\forall k\in\mathbb{N}:$ $k\ge K'$ $$Ck^{c_2}<{\left(\frac{4}{3}\right)^{k}}.$$
Зафиксируем данное $K'$.

Таким образом, мы получаем, что $\forall a,b\in\mathbb{YF},$ $k\in\mathbb{N}$: $(a,b,k)\in\overline{P}(w,n,\delta,m)$, $k\ge K'$
$${\frac{d\left(\varepsilon,2^kb\right)d\left(2^k,a\right)}{d\left(\varepsilon,b\right)d\left(\varepsilon,a1^{|b|}\right)}}<\left(\text{при $C=2C_2e^{c_2}$, $C_2=2^{c_2}C_1$, $C_1=c_1\left(\frac{2}{\delta}+2\right)^{c_2}$}\right)<$$
$$< Ck^{c_2} {\left(\frac{1}{2}\right)^{k}}<{\left(\frac{4}{3}\right)^{k}}{\left(\frac{1}{2}\right)^{k}}={\left(\frac{2}{3}\right)^{k}}.$$

Это в свою очередь значит, что $\forall a,b\in\mathbb{YF},$ $k\in\mathbb{N}$: $(a,b,k)\in\overline{P}(w,n,\delta,m)$, $k\ge K'$
$${\frac{d\left(\varepsilon,2^kb\right)d\left(2^k,a\right)}{d\left(\varepsilon,a1^{|b|}\right)d(\varepsilon,b)}}<
{\left(\frac{2}{3}\right)^{k}}\Longleftrightarrow\text{(По Утверждению \ref{razbivaem} при $(ab),a,b\in\mathbb{YF}$)}\Longleftrightarrow$$
$$\Longleftrightarrow{\frac{d(\varepsilon,2^kb)d(2^k,a)}{d(\varepsilon,ab)}}<
{\left(\frac{2}{3}\right)^{k}}\Longleftrightarrow{\frac{d(\varepsilon,2^kb)d(2^k,a)}{d(\varepsilon,w_m)}}<
{\left(\frac{2}{3}\right)^{k}}.$$

Пусть $m,n\in\mathbb{N}_0:$ $m\ge n\ge \frac{\displaystyle2 K'}{\displaystyle\delta}$. Вспомним определение $\overline{P}(w,n,\delta,m)$ и поймём, что $\forall a,b\in\mathbb{YF},$ $k\in\mathbb{N}$: $(a,b,k)\in\overline{P}(w,n,\delta,m)$
$$k> \frac{\delta n}{2}\ge \frac{\delta \frac{\displaystyle 2K'}{\displaystyle\delta}}{\displaystyle2}=\frac{{\displaystyle 2K'}}{\displaystyle2}=K'.$$

Таким образом, мы получаем, что если $m,n\in\mathbb{N}_0:$ $m\ge n\ge \frac{\displaystyle 2 K'}{\displaystyle\delta}$, то $\forall a,b\in\mathbb{YF},$ $k\in\mathbb{N}$: $(a,b,k)\in\overline{P}(w,n,\delta,m)$
$${\frac{d(\varepsilon,2^kb)d(2^k,a)}{d(\varepsilon,w_m)}}<
{\left(\frac{2}{3}\right)^{k}}.$$

Давайте просуммируем данное выражение по $(a,b,k)\in\overline{P}(w,n,\delta,m)$. Для начала зафиксируем $a$ и $b$ и просуммируем по $k\in\mathbb{N}:$ $k>\frac{\delta n}{2}$:
$$\sum_{k\in\mathbb{N}: \; k>\frac{\delta n}{2}}{\frac{d\left(\varepsilon,2^kb\right)d\left(2^k,a\right)}{d(\varepsilon,w_m)}}<\sum_{k\in\mathbb{N}: \; k>\frac{\delta n}{2}}
{\left(\frac{2}{3}\right)^{k}}= \sum_{k=\left\lfloor\frac{\delta n}{2}\right\rfloor+1}^{\infty}
{\left(\frac{2}{3}\right)^{k}}= \left(\frac{2}{3}\right)^{\left\lfloor\frac{\delta n}{2}\right\rfloor+1} \sum_{k=0}^{\infty}
{\left(\frac{2}{3}\right)^{k}}=3\left(\frac{2}{3}\right)^{\lfloor\frac{\delta n}{2}\rfloor+1}.$$

Осталось просуммировать данное выражение по всем парам $(a,b)\in\mathbb{YF}^2:$ $ab=w_m$ и $|b|< (1-\delta)n$. Несложно заметить, что наше выражение не зависит от $a$ и $b$, а количество данных разбиений $w_m$ точно не больше, чем $n$, так как $|b|<(1-\delta)n< n\Longrightarrow |b|\in \overline{n-1}$.

А значит, если $m,n\in\mathbb{N}_0:$ $m\ge n\ge \frac{\displaystyle 2 K'}{\displaystyle\delta}$, то
$$ \sum_{(a,b,k)\in\overline{P}(w,n,\delta,m) }{\frac{d\left(\varepsilon,2^kb\right)d\left(2^k,a\right)}{d(\varepsilon,w_m)}}\le n\cdot\sum_{k\in\mathbb{N}_0: \; k>\frac{\delta n}{2}}{\frac{d\left(\varepsilon,2^kb\right)d\left(2^k,a\right)}{d(\varepsilon,w_m)}} <$$
$$<n\cdot3 \left(\frac{2}{3}\right)^{\left\lfloor\frac{\delta n}{2}\right\rfloor+1}=3n\left(\frac{2}{3}\right)^{\left\lfloor\frac{\delta n}{2}\right\rfloor+1}\le 3n \left(\frac{2}{3}\right)^{\frac{\delta n}{2}}= 3n \left(\left(\frac{2}{3}\right)^{\frac{\delta }{2}}\right)^n.$$

Итак, таким образом, мы доказали, что $\exists N=K'\in\mathbb{N}$: $\forall m,n\in\mathbb{N}_0:$ $m\ge n \ge N$
$$ \sum_{(a,b,k)\in\overline{P}(w,n,\delta,m) }{\frac{d\left(\varepsilon,2^kb\right)d\left(2^k,a\right)}{d(\varepsilon,w_m)}}< 3n \left(\left(\frac{2}{3}\right)^{\frac{\delta }{2}}\right)^n.$$

А это значит, что $\forall n\in\mathbb{N}_0:$ $ n\ge N$
$${\sum_{v\in \overline{P}(w,n,\delta)}\mu_w(v)}={\sum_{v\in\overline{P}(w,n,\delta)}\left(\lim_{m\to\infty}\frac{d(\varepsilon,v)d(v,w_m)}{d(\varepsilon,w_m)}\right)}=\lim_{m\to\infty}{\left(\sum_{v\in\overline{P}(w,n,\delta)}\frac{d(\varepsilon,v)d(v,w_m)}{d(\varepsilon,w_m)}\right)}\le$$
$$ \le\text{ (По Лемме \ref{ass} при $w\in\mathbb{YF}_\infty$, $\delta\in(0,1)$, $n,m\in\mathbb{N}_0$) }\le$$
$$\le \lim_{m\to\infty}\left( \sum_{(a,b,k)\in\overline{P}(w,n,\delta,m)}{\frac{d(\varepsilon,2^kb)d(2^k,a)}{d(\varepsilon,w_m)}}\right)\le 3n \left(\left(\frac{2}{3}\right)^{\frac{\delta }{2}}\right)^n.$$

Ясно, что $\forall w\in\mathbb{YF}_\infty$, $v\in\mathbb{YF}$
$$\mu_w(v)=\lim_{m \to \infty} \frac{d(\varepsilon,v)d(v,w_m)}{d(\varepsilon,w_m)}\ge 0 ,$$
а значит, при наших $w\in\mathbb{YF}_\infty^+,$ $\delta\in(0,1)$ $\forall n\in\mathbb{N}_0:$ $n\ge N$
$$ 3n \left(\left(\frac{2}{3}\right)^{\frac{\delta }{2}}\right)^n\ge  {\sum_{v\in \overline{P}(w,n,\delta)}\mu_w(v)}\ge 0.$$

Ясно, что
 $$\lim_{n\to\infty}\left( 3n \left(\left(\frac{2}{3}\right)^{\frac{\delta }{2}}\right)^n\right)=\lim_{n\to\infty}\left(\frac{\displaystyle 3n}{\displaystyle  \left(\left(\frac{\displaystyle 3}{\displaystyle 2}\right)^{\frac{\delta }{2}}\right)^n}\right)=0,$$
    так как экспонента растёт быстрее многочлена.

А значит, по Лемме о двух полицейских, 
$$\lim_{n \to \infty}{\sum_{v\in \overline{P}(w,n,\delta)}\mu_w(v)}=0,$$
что доказывает первый пункт.

Также заметим, что
\begin{itemize}
    \item $$\overline{P}(w,n,\delta)\cup P(w,n,\delta)=\left\{v\in\mathbb{YF}_n:\; {h'(v,w)}< (1 - \delta)n\right\}\cup \left\{v\in\mathbb{YF}_n:\; {h'(v,w)}\ge (1 - \delta)n\right\}=\mathbb{YF}_n;$$
    \item  $$\overline{P}(w,n,\delta)\cap P(w,n,\delta)=\left\{v\in\mathbb{YF}_n: \;{h'(v,w)}< (1 - \delta)n\right\}\cap \left\{v\in\mathbb{YF}_n:\; {h'(v,w)}\ge (1 - \delta)n\right\}=\varnothing;$$
    \item (Следствие \ref{mera1}) $\forall w\in\mathbb{YF}_\infty$ и $n\in\mathbb{N}_0$
    $${\sum_{v\in \mathbb{YF}_n}\mu_w(v)}=1.$$
\end{itemize}

Таким образом, можно сделать вывод, что
$$\lim_{n \to \infty}{\sum_{v\in P(w,n,\delta)}\mu_w(v)}=1,$$
то есть второй пункт доказан.

Оба пункта доказаны.

Теорема доказана.

\end{proof}

\newpage

\section{Доказательство Следствия \ref{main1}}

\begin{Oboz}
Пусть $w\in\mathbb{YF}_\infty$, $n,l\in\mathbb{N}_0$. Тогда
\begin{itemize}
    \item $$Q(w,n,l):=\left\{v\in\mathbb{YF}_n:\; {h'(v,w)}\ge l\right\};$$
    \item $$ \overline{Q}(w,n,l):=\left\{v\in\mathbb{YF}_n:\; {h'(v,w)}< l\right\}.$$
\end{itemize}
\end{Oboz}

\renewcommand{\labelenumi}{\arabic{enumi}$)$}
\renewcommand{\labelenumii}{\arabic{enumi}.\arabic{enumii}$^\circ$}
\renewcommand{\labelenumiii}{\arabic{enumi}.\arabic{enumii}.\arabic{enumiii}$^\circ$}

\begin{Col}[Из Теоремы \ref{t1}] \label{main1}
Пусть $w\in\mathbb{YF}_\infty^+$, $l\in\mathbb{N}_0$. Тогда
\begin{enumerate}
    \item $$ \lim_{n \to \infty}{\sum_{v\in \overline{Q}(w,n,l)}\mu_w(v)=0};$$
    \item $$\lim_{n \to \infty}{\sum_{v\in Q(w,n,l)}\mu_w(v)=1}.$$
\end{enumerate}

\renewcommand{\labelenumi}{\arabic{enumi}$^\circ$}
\renewcommand{\labelenumii}{\arabic{enumi}.\arabic{enumii}$^\circ$}
\renewcommand{\labelenumiii}{\arabic{enumi}.\arabic{enumii}.\arabic{enumiii}$^\circ$}

\end{Col}
\begin{proof}
Давайте применим Теорему \ref{t1} (первый пункт) при $\delta=\frac{1}{2}$ и том же $w\in\mathbb{YF}_\infty^+$:
$$\lim_{n \to \infty}{\sum_{v\in \overline{P}\left(w,n,\frac{1}{2}\right)}\mu_w(v)=0}.$$
    
Несложно заметить, что при $n\in\mathbb{N}_0:n\ge 2l$
$$\overline{P}\left(w,n,\frac{1}{2}\right)=\left\{v\in\mathbb{YF}_n:\; {h'(v,w)}< \left(1-\frac{1}{2}\right)n\right\}=\left\{v\in\mathbb{YF}_n:\; {h'(v,w)}< \frac{1}{2}n\right\}\supseteq$$
$$\supseteq \left\{v\in\mathbb{YF}_n:\; {h'(v,w)}< l\right\}=\overline{Q}(w,n,l).$$

Ясно, что $\forall w\in\mathbb{YF}_\infty$, $v\in\mathbb{YF}$
$$\mu_w(v)=\lim_{m \to \infty} \frac{d(\varepsilon,v)d(v,w_m)}{d(\varepsilon,w_m)}\ge 0,$$
а значит, при наших $w\in\mathbb{YF}_\infty^+,$ $l\in\mathbb{N}_0$ $\forall n\in\mathbb{N}_0:$ $n\ge 2l$
$${\sum_{v\in \overline{P}\left(w,n,\frac{1}{2}\right)}\mu_w(v)}\ge {\sum_{v\in \overline{Q}(w,n,l)}\mu_w(v)} \ge 0.$$

А значит, по Лемме о двух полицейских, 
$$\lim_{n \to \infty}{\sum_{v\in \overline{Q}(w,n,l)}\mu_w(v)}=0,$$
что доказывает первый пункт.

Также заметим, что
\begin{itemize}
    \item $$\overline{Q}\left(w,n,l\right)\cup Q\left(w,n,l\right)=\left\{v\in\mathbb{YF}_n:\; {h'(v,w)}< l\right\}\cup \left\{v\in\mathbb{YF}_n:\; {h'(v,w)}\ge l\right\}=\mathbb{YF}_n;$$
    \item $$\overline{Q}\left(w,n,l\right)\cap Q\left(w,n,l\right)=\left\{v\in\mathbb{YF}_n:\; {h'(v,w)}< l\right\}\cap \left\{v\in\mathbb{YF}_n:\; {h'(v,w)}\ge l\right\}=\varnothing;$$
    \item (Следствие \ref{mera1}) $\forall w\in\mathbb{YF}_\infty$ и $n\in\mathbb{N}_0$
    $${\sum_{v\in \mathbb{YF}_n}\mu_w(v)}=1.$$
\end{itemize}
Таким образом, можно сделать вывод, что
$$\lim_{n \to \infty}{\sum_{v\in Q(w,n,l)}\mu_w(v)}=1,$$
что доказывает второй пункт.

Таким образом, оба пункта доказаны.

Следствие доказано.

\end{proof}

\newpage

\section{Доказательство Следствия \ref{main2}}

\begin{Oboz}
Пусть $w\in\mathbb{YF}_\infty$, $n\in\mathbb{N}_0$, $\varepsilon\in\mathbb{R}_{>0}$. Тогда
\begin{itemize}
    \item $$R(w,n,\varepsilon):=\left\{v\in\mathbb{YF}_n:\; \pi(v)\in(\pi(w)(1-\varepsilon),\pi(w)(1+\varepsilon)) \right\};$$
    \item $$\overline{R}(w,n,\varepsilon):=\left\{v\in\mathbb{YF}_n:\; \pi(v)\notin(\pi(w)(1-\varepsilon),\pi(w)(1+\varepsilon)) \right\}.$$
\end{itemize}
\end{Oboz}
\begin{Col}[Из Теоремы \ref{t1}] \label{main2}
Пусть $w\in\mathbb{YF}_\infty^+$, $\varepsilon\in\mathbb{R}_{>0}$. Тогда
\begin{enumerate}
    \item $$ \lim_{n \to \infty}{\sum_{v\in \overline{R}(w,n,\varepsilon)}\mu_w(v)=0};$$
    \item $$\lim_{n \to \infty}{\sum_{v\in R(w,n,\varepsilon)}\mu_w(v)=1}.$$
\end{enumerate}
\end{Col}
\begin{proof}
Давайте докажем, что при данных $w\in\mathbb{YF}_\infty^+$ и $\varepsilon\in\mathbb{R}_{>0}$ существуют $\delta\in(0,1)$ и $N\in\mathbb{N}_0$: при любом $n\in\mathbb{N}_0:$ $n\ge N$ 
$$R(w,n,\varepsilon) \supseteq P(w,n,\delta).$$

\renewcommand{\labelenumi}{\arabic{enumi}$^\circ$}
\renewcommand{\labelenumii}{\arabic{enumi}.\arabic{enumii}$^\circ$}
\renewcommand{\labelenumiii}{\arabic{enumi}.\arabic{enumii}.\arabic{enumiii}$^\circ$}

Для этого докажем следующие утверждения:
\begin{Prop} \label{meexy1} 
Пусть $v,v',v''\in\mathbb{YF}$: $v=v'v''$. Тогда
$$\pi(v)=\pi(v'')\pi\left(v'1^{|v''|}\right).$$
\end{Prop}
\begin{proof}
Рассмотрим три случая:
\begin{enumerate}
    \item  $v''=\varepsilon$.
    
    Ясно, что в данном случае $v=v'.$ Посчитаем, воспользовавшись определением функций $\pi$ и $g$:
    $$\pi(v)=\pi(v')=\left(\prod_{i:\;g(\varepsilon,i)> 1  } \frac{g(\varepsilon,i)-1}{g(\varepsilon,i)}\right)\pi(v')=\pi(\varepsilon)\pi(v')=$$
    $$=\pi(\varepsilon)\pi\left(v'1^0\right)=\pi(\varepsilon)\pi\left(v'1^{|\varepsilon|}\right)=\pi(v'')\pi\left(v' 1^{|v''|}\right).$$
    В данном случае Утверждение доказано.
    
    \item номер $v''$ заканчивается на $1$.
    
    Ясно, что в данном случае номер $v$ тоже заканчивается на $1$, а также то, что $|v''|\ge 1$. Посчитаем, воспользовавшись определением функций $\pi$ и $g$:
    $$\pi(v)=\prod_{i:\;g(v,i)> 1  } \frac{g(v,i)-1}{g(v,i)}=\prod_{i=1}^{d(v)} \frac{g(v,i)-1}{g(v,i)}= $$
    $$=\left(\prod_{i=1}^{d(v'')} \frac{g(v,i)-1}{g(v,i)}\right)\left(\prod_{i=d(v'')+1}^{d(v)} \frac{g(v,i)-1}{g(v,i)}\right)=$$
    $$=\left(\prod_{i=1}^{d(v'')} \frac{g(v'',i)-1}{g(v'',i)}\right)\left(\prod_{i=1}^{d\left(v'1^{|v''|}\right)} \frac{g(v'1^{|v''|},i)-1}{g\left(v'1^{|v''|},i\right)}\right)=$$
    $$=\left(\prod_{i:\;g(v'',i)> 1  } \frac{g(v'',i)-1}{g(v'',i)}\right)\left(\prod_{i:\;g\left(v'1^{|v''|},i\right)> 1  } \frac{g(v'1^{|v''|},i)-1}{g\left(v'1^{|v''|},i\right)}\right)=$$
    $$=\pi(v'')\pi\left(v'1^{|v''|}\right).$$
    В данном случае Утверждение доказано.
    
    \item номер $v''$ заканчивается на $2$.
    
    Ясно, что в данном случае номер $v$ тоже заканчивается на $2$, а также то, что $|v''|\ge 2$. Посчитаем, воспользовавшись определением функций $\pi$ и $g$:
    $$\pi(v)=\prod_{i:\;g(v,i)> 1  } \frac{g(v,i)-1}{g(v,i)}=\prod_{i=2}^{d(v)} \frac{g(v,i)-1}{g(v,i)}= $$
    $$=\left(\prod_{i=2}^{d(v'')} \frac{g(v,i)-1}{g(v,i)}\right)\left(\prod_{i=d(v'')+1}^{d(v)} \frac{g(v,i)-1}{g(v,i)}\right)=$$
    $$=\left(\prod_{i=2}^{d(v'')} \frac{g(v'',i)-1}{g(v'',i)}\right)\left(\prod_{i=1}^{d\left(v'1^{|v''|}\right)} \frac{g(v'1^{|v''|},i)-1}{g\left(v'1^{|v''|},i\right)}\right)=$$
    $$=\left(\prod_{i:\;g(v'',i)> 1  } \frac{g(v'',i)-1}{g(v'',i)}\right)\left(\prod_{i:\;g\left(v'1^{|v''|},i\right)> 1  } \frac{g(v'1^{|v''|},i)-1}{g\left(v'1^{|v''|},i\right)}\right)=$$
    $$=\pi(v'')\pi\left(v'1^{|v''|}\right).$$
    В данном случае Утверждение доказано.

\end{enumerate}

Все случаи разобраны.

Утверждение доказано.

\end{proof}

\begin{Prop}\label{pashaplohoi}
Пусть $w\in\mathbb{YF}_\infty^+$, $\delta\in(0,1)$, $\varepsilon\in\mathbb{R}_{>0}$. Тогда $\exists N''\in\mathbb{N}_0$: $\forall n\in\mathbb{N}_0,$ $v\in\mathbb{YF}:$ $n\ge N''$, $v\in P(w,n,\delta)$
$$\frac{\pi(v)}{\pi(w)} < 1+\varepsilon.$$
\end{Prop}
\begin{proof}
Ясно, что то, что $w\in\mathbb{YF}_\infty^+$ означает, что выражение $\frac{\displaystyle\pi(v)}{\displaystyle\pi(w)}$ определено.

Давайте считать, что при нашем $w\in\mathbb{YF}_\infty^+$ и произвольно выбранной вершине $v\in\mathbb{YF}$ $$w=w'w'';\quad v=v'w'',$$
где $w''$ -- самый длинный общий суффикс $v$ и $w$. Ясно, что тут $w'\in\mathbb{YF}_\infty$, $v',w''\in\mathbb{YF}$. Кроме того, из определения функции $h'$ ясно, что  $|w''|=h'(v,w)$.

Значит,
$$\frac{\pi(v)}{\pi(w)}=\left(\text{По Утверждению \ref{meexy1} при $v,v',w''\in\mathbb{YF}$, а также при $w,w',w''\in\mathbb{YF}$}\right)=$$
$$=\frac{\pi\left(w''\right)\pi\left(v'1^{\left|w''\right|}\right)}{\pi\left(w''\right)\pi\left(w'1^{\left|w''\right|}\right)}=\frac{\pi\left(v'1^{\left|w''\right|}\right)}{\pi\left(w'1^{\left|w''\right|}\right)}\le\left(\text{По Замечанию \ref{promezhutok} при $\left(v'1^{\left|w''\right|}\right)\in\mathbb{YF}$}\right) \le $$
$$\le\frac{1}{\pi\left(w'1^{\left|w''\right|}\right)}=\frac{\displaystyle1}{\displaystyle \prod_{i:\;g\left(w'1^{\left|w''\right|},i\right)>1 } \frac{ g\left(w'1^{\left|w''\right|},i\right)-1}{g\left(w'1^{\left|w''\right|},i\right)}}=$$
$$=(\text{Следует из определения функции $g$})=\frac{\displaystyle1}{\displaystyle \prod_{i:\;g(w,i)>\max\left(1,\left|w''\right|\right)  } \frac{ g(w,i)-1}{g(w,i)}}.$$

Мы знаем, что $w\in\mathbb{YF}_\infty^+$, а это значит, что 
$$\prod_{i:\;g(w,i)>1  } \frac{ g(w,i)-1}{g(w,i)}>0,$$
а это значит, что $\forall \varepsilon'\in\mathbb{R}_{>0}$ $\exists n'\in\mathbb{N}_0:$ $n'\ge 1$ и
$$\prod_{i:\;g(w,i)>n'  } \frac{ g(w,i)-1}{g(w,i)}>1-\varepsilon',$$
а это, в свою очередь, значит, что при нашем $\varepsilon\in\mathbb{R}_{>0}$ существует $ n'\in\mathbb{N}_0:n'\ge 1$ и
$$\frac{\displaystyle1}{\displaystyle\prod_{i:\;g(w,i)>n'  } \frac{ g(w,i)-1}{g(w,i)}}<1+\varepsilon.$$

Пусть $N''= \left\lceil\frac{n'}{(1-\delta)}\right\rceil\in\mathbb{N}_0$. Тогда если $v\in P(w,n,\delta)$ при $n\in\mathbb{N}_0:$ $n\ge N''$, то, по определению $P(w,n,\delta)$, $$|w''|=h'(v,w)\ge (1-\delta)n\ge (1-\delta)N''= (1-\delta)\left\lceil\frac{n'}{(1-\delta)}\right\rceil\ge (1-\delta)\frac{n'}{(1-\delta)}=n',$$
то есть $$\frac{\displaystyle1}{\displaystyle\prod_{i:\;g(w,i)>|w''|  } \frac{ g(w,i)-1}{g(w,i)}}\le\frac{\displaystyle1}{\displaystyle\prod_{i:\;g(w,i)>n'  } \frac{ g(w,i)-1}{g(w,i)}}< 1+\varepsilon,$$
а значит, если $v\in P(w,n,\delta)$ при $n\in\mathbb{N}_0:n\ge N''$, то
$$\frac{\pi(v)}{\pi(w)}< 1+\varepsilon,$$
что и требовалось.

Утверждение доказано.
\end{proof}

\begin{Prop} \label{podotriopyat}
Пусть  $w\in\mathbb{YF}_\infty^+$, $\delta\in(0,1)$, $n\in\mathbb{N}_0,$  $v\in\mathbb{YF}:$ $n\ge 1,$ $v\in P(w,n,\delta)$. Тогда
$${\displaystyle \prod_{i=1}^{\left\lceil\frac{\delta n}{2}\right\rceil} \frac{ g\left(2^{\left\lceil\frac{\delta n
}{2}\right\rceil} 1^{2\left\lceil\frac{(1-\delta)n}{2}\right\rceil-1},i\right)-1}{g\left(2^{\left\lceil\frac{\delta n
}{2}\right\rceil} 1^{2\left\lceil\frac{(1-\delta)n}{2}\right\rceil-1},i\right)}}
\le\frac{\pi(v)}{\pi(w)}.$$
\end{Prop}
\begin{proof}
Ясно, что то, что $w\in\mathbb{YF}_\infty^+$ означает, что выражение $\frac{\displaystyle\pi(v)}{\displaystyle\pi(w)}$ определено.

Давайте считать, что при нашем $w\in\mathbb{YF}_\infty^+$ и произвольно выбранной вершине $v\in\mathbb{YF}$ $$w=w'w'';\quad v=v'w'',$$
где $w''$ -- самый длинный общий суффикс $v$ и $w$. Ясно, что тут $w'\in\mathbb{YF}_\infty$, $v',w''\in\mathbb{YF}$. Кроме того, из определения функции $h'$ ясно, что  $|w''|=h'(v,w)$.

Значит
$$\frac{\pi(v)}{\pi(w)}=\left(\text{По Утверждению \ref{meexy1} при $v,v',w''\in\mathbb{YF}$, а также при $w,w',w''\in\mathbb{YF}$}\right)=$$
$$=\frac{\pi\left(w''\right)\pi\left(v'1^{\left|w''\right|}\right)}{\pi\left(w''\right)\pi\left(w'1^{\left|w''\right|}\right)}=\frac{\pi\left(v'1^{\left|w''\right|}\right)}{\pi\left(w'1^{\left|w''\right|}\right)}\ge\left(\text{По Замечанию \ref{promezhutok} при $\left(w'1^{\left|w''\right|}\right)\in\mathbb{YF}$}\right) \ge $$
$$\ge\pi\left(v'1^{\left|w''\right|}\right)={\displaystyle \prod_{i:\;g\left(v'1^{\left|w''\right|},i\right)>1 } \frac{ g\left(v'1^{|w''|},i\right)-1}{g\left(v'1^{\left|w''\right|},i\right)}}.$$

Вспомним, что $v\in P(w,n,\delta)$ при $w\in\mathbb{YF}_\infty^+$, $n\in\mathbb{N}_0:$ $n\ge 1$ и $\delta\in(0,1)$. А это значит, что
$$|w''|=h'(v,w)\ge (1-\delta)n>0,$$
то есть (так как $|w''|\in\mathbb{N}_0$) $$|w''|\ge 1.$$
Таким образом, из определения функции $g$ ясно, что $\forall i\in\left\{1,\ldots,d\left(v'1^{\left|w''\right|}\right)\right\}$
$$g\left(v'1^{|w''|},i\right)\ge 2.$$
Таким образом, наше выражение равняется следующему:
$${\displaystyle \prod_{i=1}^{d\left(v'1^{\left|w''\right|}\right)} \frac{ g\left(v'1^{\left|w''\right|},i\right)-1}{g\left(v'1^{\left|w''\right|},i\right)}}={\displaystyle \prod_{i=1}^{d\left(v'\right)} \frac{ g\left(v'1^{\left|w''\right|},i\right)-1}{g\left(v'1^{\left|w''\right|},i\right)}}.$$

Вспомним, что $v\in P(w,n,\delta)$ при $w\in\mathbb{YF}_\infty^+$, $n\in\mathbb{N}_0:$ $n\ge 1$ и $\delta\in(0,1)$. А это значит, что
$$\left|w''\right|=h'(v,w)\ge (1-\delta)n>0.$$

Также ясно, что
$$(1-\delta)n+1= 2\frac{(1-\delta)n}{2}+2-1 >2\left\lceil\frac{(1-\delta)n}{2}\right\rceil-1\ge 2-1=1\Longrightarrow$$
$$\Longrightarrow\left(\text{так как }\left(2\left\lceil\frac{(1-\delta)n}{2}\right\rceil-1\right)\in\mathbb{Z}\right)\Longrightarrow (1-\delta)n\ge 2\left\lceil\frac{(1-\delta)n}{2}\right\rceil-1\ge1.$$

Делаем вывод, что $$|w''|=h'(v,w)\ge(1-\delta)n\ge 2\left\lceil\frac{(1-\delta)n}{2}\right\rceil-1\ge1.$$
Таким образом, из определения функции $g$ ясно, что $\forall i\in\left\{1,\ldots,d\left(v'1^{\left|w''\right|}\right)\right\}=\left\{1,\ldots,d\left(v'\right)\right\}=\left\{1,\ldots,d\left(v'1^{2\left\lceil\frac{(1-\delta)n}{2}\right\rceil-1}\right)\right\}$
$$g\left(v'1^{\left|w''\right|},i\right)\ge g\left(v'1^{2\left\lceil\frac{(1-\delta)n}{2}\right\rceil-1},i\right)\ge 2.$$
Таким образом, делаем вывод, что
$${\displaystyle \prod_{i=1}^{d\left(v'\right)} \frac{ g\left(v'1^{\left|w''\right|},i\right)-1}{g\left(v'1^{\left|w''\right|},i\right)}}\ge {\displaystyle \prod_{i=1}^{d\left(v'\right)} \frac{ g\left(v'1^{2\left\lceil\frac{(1-\delta)n}{2}\right\rceil-1},i\right)-1}{g\left(v'1^{2\left\lceil\frac{(1-\delta)n}{2}\right\rceil-1},i\right)}}\ge $$
$$\ge(\text{По определению функции $g$})\ge  {\displaystyle \prod_{i=1}^{\left\lceil\frac{|v'|}{2}\right\rceil} \frac{ g\left(2^{\left\lceil\frac{|v'|}{2}\right\rceil} 1^{2\left\lceil\frac{(1-\delta)n}{2}\right\rceil-1},i\right)-1}{g\left(2^{\left\lceil\frac{|v'|}{2}\right\rceil} 1^{2\left\lceil\frac{(1-\delta)n}{2}\right\rceil-1},i\right)}}.$$
Вспомним, что $v\in P(w,n,\delta)$ при $w\in\mathbb{YF}_\infty^+$, $n\in\mathbb{N}_0:$ $n\ge 1$ и $\delta\in(0,1)$. А это значит, что
$$|v'|=n-|w''|=n-h'(v,w)\le n- (1-\delta)n=\delta n\Longrightarrow\frac{|v'|}{2} \le \frac{ \delta n}{2}\Longrightarrow \left\lceil\frac{|v'|}{2}\right\rceil \le \left\lceil\frac{\delta n}{2}\right\rceil,$$
что, в свою очередь, значит, что
$${\displaystyle \prod_{i=1}^{\left\lceil\frac{|v'|}{2}\right\rceil} \frac{ g\left(2^{\left\lceil\frac{|v'|}{2}\right\rceil} 1^{2\left\lceil\frac{(1-\delta)n}{2}\right\rceil-1},i\right)-1}{g\left(2^{\left\lceil\frac{|v'|}{2}\right\rceil} 1^{2\left\lceil\frac{(1-\delta)n}{2}\right\rceil-1},i\right)}}\ge {\displaystyle \prod_{i=1}^{\left\lceil\frac{\delta n}{2}\right\rceil} \frac{ g\left(2^{\left\lceil\frac{\delta n
}{2}\right\rceil} 1^{2\left\lceil\frac{(1-\delta)n}{2}\right\rceil-1},i\right)-1}{g\left(2^{\left\lceil\frac{\delta n
}{2}\right\rceil} 1^{2\left\lceil\frac{(1-\delta)n}{2}\right\rceil-1},i\right)}}.$$

Таким образом, мы доказали, что
$$\frac{\pi(v)}{\pi(w)}\ge{\displaystyle \prod_{i=1}^{\left\lceil\frac{\delta n}{2}\right\rceil} \frac{ g\left(2^{\left\lceil\frac{\delta n
}{2}\right\rceil} 1^{2\left\lceil\frac{(1-\delta)n}{2}\right\rceil-1},i\right)-1}{g\left(2^{\left\lceil\frac{\delta n
}{2}\right\rceil} 1^{2\left\lceil\frac{(1-\delta)n}{2}\right\rceil-1},i\right)}},$$
что и требовалось.

Утверждение доказано.
\end{proof}
\begin{Lemma}\label{truefrick}
   Пусть $w\in\mathbb{YF}_\infty^+$, $\varepsilon\in\mathbb{R}_{>0}$. Тогда $\exists N'\in \mathbb{N}_0$, $\delta\in(0,1):$ $\forall n\in\mathbb{N}_0,$ $v\in\mathbb{YF}:$ $n\ge N',$ $v\in P(w,n,\delta)$
    $$1-\varepsilon<\frac{\pi(v)}{\pi(w)}.$$
\end{Lemma}
\begin{proof}
По Утверждению \ref{podotriopyat} при нашем $w\in\mathbb{YF}_\infty^+$ и произвольных $n\in\mathbb{N}_0,$ $\delta\in(0,1)$, $v\in\mathbb{YF}:$ $n\ge 1,$ $ v\in P(w,n,\delta)$ 
$$ {\displaystyle \prod_{i=1}^{\left\lceil\frac{\delta n}{2}\right\rceil} \frac{ g\left(2^{\left\lceil\frac{\delta n
}{2}\right\rceil} 1^{2\left\lceil\frac{(1-\delta)n}{2}\right\rceil-1},i\right)-1}{g\left(2^{\left\lceil\frac{\delta n
}{2}\right\rceil} 1^{2\left\lceil\frac{(1-\delta)n}{2}\right\rceil-1},i\right)}}\le\frac{\pi(v)}{\pi(w)}.$$

Пусть $\delta=\min\left(\frac{\varepsilon}{2},\frac{1}{2}\right)$. Ясно, что $\delta\in(0,1)$.

Теперь введём следующие обозначения: 
\begin{itemize}
    \item $$a=a(n)=\left\lceil\frac{(1-\delta)n}{2}\right\rceil;$$
    \item $$b=b(n)=\left\lceil\frac{\delta n}{2}\right\rceil.$$
\end{itemize}

Ясно, что 
\begin{itemize}
    \item $$a(n)=\left\lceil\frac{(1-\delta)n}{2}\right\rceil\ge \frac{(1-\delta)n}{2}\xrightarrow{n\to\infty}\infty\Longrightarrow a(n)\xrightarrow{n\to \infty} \infty;$$
    \item $$b(n)=\left\lceil\frac{\delta n}{2}\right\rceil\ge \frac{\delta n}{2}\xrightarrow{n\to\infty}\infty\Longrightarrow b(n)\xrightarrow{n\to \infty}\infty.$$
\end{itemize}

Утверждение \ref{podotriopyat} при нашем $w\in\mathbb{YF}_\infty^+$ и выбранном $\delta\in(0,1)$ в новых обозначениях принимает следующий вид (помним про определение функции $g$):

При произвольных $n\in\mathbb{N}_0,$ $v\in\mathbb{YF}:$ $n\ge 1,$ $v\in P(w,n,\delta)$ 
$${\displaystyle \prod_{i=a(n)}^{a(n)+b(n)-1}\frac{2i-1}{2i}}\le \frac{\pi(v)}{\pi(w)}.$$

Давайте найдём
$$\lim_{n\to\infty}\left(\prod_{i=a(n)}^{a(n)+b(n)-1}\frac{2i-1}{2i}\right).$$
    
Посчитаем, держа в голове, что $a(n)\xrightarrow{n \to \infty}\infty$ и $b(n)\xrightarrow{n \to \infty}\infty,$ что в частности значит, что мы можем рассматривать только такие $n$, что $a(n)\ge 2$ и $b(n)\ge 1$:
$$\lim_{n\to\infty}\left(\prod_{i=a}^{a+b-1}\frac{2i-1}{2i}\right)=\lim_{n\to\infty}\left(\prod_{i=a}^{a+b-1}\frac{(2i-1)2i}{(2i)^2}\right)=\lim_{n\to\infty}\left(\frac{\displaystyle\prod_{i=2a-1}^{2a+2b-2}i}{\displaystyle\left(2^b\prod_{i=a}^{a+b-1}i\right)^2}\right)=$$
$$=\lim_{n\to\infty}\left(\frac{\displaystyle\frac{(2a+2b-2)!}{(2a-2)!}}{\displaystyle2^{2b}\left(\frac{(a+b-1)!}{(a-1)!}\right)^2}\right)=\lim_{n\to\infty}\left(\frac{1}{2^{2b}}\left(\frac{(a-1)!}{(a+b-1)!}\right)^2\frac{(2a+2b-2)!}{(2a-2)!}\right).$$

По формуле Стирлинга данное выражение равняется следующему выражению при некоторых $\theta_{a+b-1},\theta_{a-1},\theta_{2a-2},\theta_{2a+2b-2}\in(0,1)$ (тут важно, что мы рассматривает только такие $n$, что $a(n)\ge 2$ и $b(n)\ge 1$, что, в свою очередь, значит, что $a+b-1\ge 1;$ $a-1\ge 1;$ $2a-2\ge 1;$ $2a+2b-2\ge 1$):
$$\lim_{n\to\infty}\left(\frac{1}{2^{2b}}\frac{{2\pi(a-1)}\left(\frac{(a-1)}{e}\right)^{2a-2}\left(\exp{\frac{\theta_{a-1}}{12(a-1)}}\right)^2}{{2\pi(a+b-1)}\left(\frac{(a+b-1)}{e}\right)^{2a+2b-2}\left(\exp{\frac{\theta_{a+b-1}}{12(a+b-1)}}\right)^2}\cdot\right.$$
$$\left.
\cdot\frac{\sqrt{2\pi(2a+2b-2)}\left(\frac{(2a+2b-2)}{e}\right)^{2a+2b-2}\exp{\frac{\theta_{2a+2b-2}}{12(2a+2b-2)}}}{\sqrt{2\pi(2a-2)}\left(\frac{(2a-2)}{e}\right)^{2a-2}\exp{\frac{\theta_{2a-2}}{12(2a-2)}}}\right)=$$
$$=\lim_{n\to\infty}\left(\frac{1}{2^{2b}}\frac{{(a-1)}\left({a-1}\right)^{2a-2}\left(\exp{\frac{\theta_{a-1}}{12(a-1)}}\right)^2}{{(a+b-1)}\left({a+b-1}\right)^{2a+2b-2}\left(\exp{\frac{\theta_{a+b-1}}{12(a+b-1)}}\right)^2}\cdot\right.$$$$\left.\cdot\frac{\sqrt{(2a+2b-2)}\left({2a+2b-2}\right)^{2a+2b-2}\exp{\frac{\theta_{2a+2b-2}}{12(2a+2b-2)}}}{\sqrt{(2a-2)}\left({2a-2}\right)^{2a-2}\exp{\frac{\theta_{2a-2}}{12(2a-2)}}}\right).$$

Для начала рассмотрим
$$\frac{\left(\exp{\frac{\theta_{a-1}}{12(a-1)}}\right)^2}{\left(\exp{\frac{\theta_{a+b-1}}{12(a+b-1)}}\right)^2}\frac{\exp{\frac{\theta_{2a+2b-2}}{12(2a+2b-2)}}}{\exp{\frac{\theta_{2a-2}}{12(2a-2)}}}.$$
Мы уже поняли, что $a(n)\xrightarrow{n \to \infty}\infty$ и $b(n)\xrightarrow{n \to \infty}\infty,$ а это значит, что (помним, что если $ i\in\mathbb{N}$, то $\theta_i\in(0,1)$)
\begin{itemize}
    \item $$\lim_{n\to\infty}\left(a(n)-1\right)=\infty \Longrightarrow\lim_{n\to\infty}{\frac{\theta_{a-1}}{12(a-1)}}=0\Longleftrightarrow \lim_{n\to\infty}{\left(\exp\frac{\theta_{a-1}}{12(a-1)}\right)}=1;$$
    \item $$\lim_{n\to\infty}\left(a(n)+b(n)-1\right)=\infty \Longrightarrow\lim_{n\to\infty}{\frac{\theta_{a+b-1}}{12(a+b-1)}}=0\Longleftrightarrow \lim_{n\to\infty}{\left(\exp\frac{\theta_{a+b-1}}{12(a+b-1)}\right)}=1;$$
    \item $$\lim_{n\to\infty}\left(2a(n)+2b(n)-2\right)=\infty \Longrightarrow\lim_{n\to\infty}{\frac{\theta_{2a+2b-2}}{12(2a+2b-2)}}=0\Longleftrightarrow \lim_{n\to\infty}{\left(\exp\frac{\theta_{2a+2b-2}}{12(2a+2b-2)}\right)}=1;$$
    \item $$\lim_{n\to\infty}\left(2a(n)-2\right)=\infty \Longrightarrow\lim_{n\to\infty}{\frac{\theta_{2a-2}}{12(2a-2)}}=0\Longleftrightarrow \lim_{n\to\infty}{\left(\exp\frac{\theta_{2a-2}}{12(2a-2)}\right)}=1.$$
    А из этого следует, что
\end{itemize}
$$\lim_{n\to\infty}\frac{\left(\exp{\frac{\theta_{a-1}}{12(a-1)}}\right)^2}{\left(\exp{\frac{\theta_{a+b-1}}{12(a+b-1)}}\right)^2}\frac{\exp{\frac{\theta_{2a+2b-2}}{12(2a+2b-2)}}}{\exp{\frac{\theta_{2a-2}}{12(2a-2)}}}=1.$$

Таким образом, мы поняли, что
$$\lim_{n\to\infty}\left(\frac{1}{2^{2b}}\frac{{(a-1)}\left({a-1}\right)^{2a-2}\left(\exp{\frac{\theta_{a-1}}{12(a-1)}}\right)^2}{{(a+b-1)}\left({a+b-1}\right)^{2a+2b-2}\left(\exp{\frac{\theta_{a+b-1}}{12(a+b-1)}}\right)^2}\cdot\right.$$
$$\left.\cdot\frac{\sqrt{(2a+2b-2)}\left({2a+2b-2}\right)^{2a+2b-2}\exp{\frac{\theta_{2a+2b-2}}{12(2a+2b-2)}}}{\sqrt{(2a-2)}\left({2a-2}\right)^{2a-2}\exp{\frac{\theta_{2a-2}}{12(2a-2)}}}\right)=$$
$$=\lim_{n\to\infty}\left(\frac{1}{2^{2b}}\cdot\frac{{(a-1)}\left({a-1}\right)^{2a-2}}{{(a+b-1)}\left({a+b-1}\right)^{2a+2b-2}}\cdot\frac{\sqrt{(2a+2b-2)}\left({2a+2b-2}\right)^{2a+2b-2}}{\sqrt{(2a-2)}\left({2a-2}\right)^{2a-2}}\right)=$$
$$=\lim_{n\to\infty}\left(\frac{{(a-1)}\left({a-1}\right)^{2a-2}}{{(a+b-1)}\left({a+b-1}\right)^{2a+2b-2}}\cdot\frac{\sqrt{(a+b-1)}\left({a+b-1}\right)^{2a+2b-2}}{\sqrt{(a-1)}\left({a-1}\right)^{2a-2}}\right)=$$
$$=\lim_{n\to\infty}\left(\frac{{(a-1)}}{{(a+b-1)}}\cdot\frac{\sqrt{(a+b-1)}}{\sqrt{(a-1)}}\right)=\lim_{n\to\infty}\left(\sqrt{\frac{a-1}{a+b-1}}\right)=\lim_{n\to\infty}\left(\sqrt{\frac{a(n)-1}{a(n)+b(n)-1}}\right)=$$
$$=\lim_{n\to\infty}\left(\sqrt{\frac{\left\lceil\frac{(1-\delta)n}{2}\right\rceil-1}{\left\lceil\frac{(1-\delta)n}{2}\right\rceil+\left\lceil\frac{\delta n}{2}\right\rceil-1}}\right).$$

Заметим, что
\begin{itemize}
    \item $$1-\frac{2}{(1-\delta)n}=\frac{(1-\delta)n-2}{(1-\delta)n}= \frac{\frac{(1-\delta)n}{2}-1}{\frac{(1-\delta)n}{2}} \le \frac{\left\lceil\frac{(1-\delta)n}{2}\right\rceil-1}{\frac{(1-\delta)n}{2}}\le  \frac{\frac{(1-\delta)n}{2}}{\frac{(1-\delta)n}{2}} =1,$$
    $$ \lim_{n\to\infty}\left(1-\frac{2}{(1-\delta)n}\right)=1\Longrightarrow \lim_{n\to\infty}\frac{\left\lceil\frac{(1-\delta)n}{2}\right\rceil-1}{\frac{(1-\delta)n}{2}}=1;$$
    \item $$1-\frac{2}{n+2}= {\frac{n}{n+2}}={\frac{\frac{n}{2}}{\frac{(1-\delta)n}{2}+\frac{\delta n}{2}+1}} \le {\frac{\frac{n}{2}}{\left\lceil\frac{(1-\delta)n}{2}\right\rceil+\left\lceil\frac{\delta n}{2}\right\rceil-1}}\le {\frac{\frac{n}{2}}{\frac{(1-\delta)n}{2}+\frac{\delta n}{2}-1}}=\frac{n}{n-2}=1+\frac{2}{n-2},$$
    $$\lim_{n\to\infty}\left(1-\frac{2}{n+2}\right)=1,\quad\lim_{n\to\infty}\left(1+\frac{2}{n-2}\right)=1\Longrightarrow \lim_{n\to\infty}\left({\frac{\frac{n}{2}}{\left\lceil\frac{(1-\delta)n}{2}\right\rceil+\left\lceil\frac{\delta n}{2}\right\rceil-1}}\right)=1.$$
\end{itemize}
    А значит, 
    $$\lim_{n\to\infty}\left(\sqrt{\frac{\left\lceil\frac{(1-\delta)n}{2}\right\rceil-1}{\left\lceil\frac{(1-\delta)n}{2}\right\rceil+\left\lceil\frac{\delta n}{2}\right\rceil-1}}\right)=\lim_{n\to\infty}\left(\sqrt{\frac{\left\lceil\frac{(1-\delta)n}{2}\right\rceil-1}{\frac{(1-\delta)n}{2}}\cdot{\frac{(1-\delta)n}{2}}\cdot\frac{\frac{n}{2}}{\left\lceil\frac{(1-\delta)n}{2}\right\rceil+\left\lceil\frac{\delta n}{2}\right\rceil-1}\cdot{\frac{2}{n}}}\right)=$$
    $$=\lim_{n\to\infty}\left(\sqrt{{\frac{(1-\delta)n}{2}}\cdot{\frac{2}{n}}}\right)=\lim_{n\to\infty}\left(1-\delta\right)=1-\delta.$$
    
    Таким образом, мы доказали, что 
    $$\lim_{n\to\infty}\left(\prod_{i=a(n)}^{a(n)+b(n)-1}\frac{2i-1}{2i}\right)=1-\delta.$$
    
    А это значит, что $\exists N'\in\mathbb{N}_0:$ $N'\ge 1,$ $\forall n\in\mathbb{N}_0:$ $n\ge N'$
    $$\prod_{i=a(n)}^{a(n)+b(n)-1}\frac{2i-1}{2i}>1-2\delta\ge 1-2\frac{\varepsilon}{2}=1-\varepsilon.$$
    
    Таким образом, мы доказали, что при наших $w\in\mathbb{YF}_\infty^+$ и $\varepsilon\in\mathbb{R}_{>0}$  $\exists N'\in\mathbb{N}_0,$ $\delta=\min\left(\frac{\varepsilon}{2},\frac{1}{2}\right)\in(0,1):$ $ \forall n\in\mathbb{N}_0,$ $v\in\mathbb{YF}:$ $n\ge N',$ $v\in P(w,n,\delta)$
    $$1-\varepsilon<\prod_{i=a(n)}^{a(n)+b(n)-1}\frac{2i-1}{2i}= {\displaystyle \prod_{i=1}^{\left\lceil\frac{\delta n}{2}\right\rceil} \frac{ g\left(2^{\left\lceil\frac{\delta n}{2}\right\rceil} 1^{2\left\lceil\frac{(1-\delta)n}{2}\right\rceil-1},i\right)-1}{g\left(2^{\left\lceil\frac{\delta n}{2}\right\rceil} 1^{2\left\lceil\frac{(1-\delta)n}{2}\right\rceil-1},i\right)}}\le\frac{\pi(v)}{\pi(w)},$$
что и требовалось.

Лемма доказана.
\end{proof}

Таким образом, по Утверждению \ref{pashaplohoi} и Лемме \ref{truefrick} при наших $w\in\mathbb{YF}_\infty^+$ и $\varepsilon\in\mathbb{R}_{>0}$ существуют $N=\max\left(N',N''\right)\in \mathbb{N}_0$ и $\delta\in(0,1)$: $\forall n\in\mathbb{N}_0,$ $v\in\mathbb{YF}:$ $n\ge N,$  $v\in P(w,n,\delta)$
$$1-\varepsilon<\frac{\pi(v)}{\pi(w)}<1+\varepsilon,$$
то есть
$$\pi(v)\in\left(\pi(w)(1-\varepsilon) ,\pi(w)(1+\varepsilon)\right),$$
то есть 
$$v\in R(w,n,\varepsilon).$$
    
    Таким образом, мы доказали, что при наших $w\in\mathbb{YF}_\infty^+$ и $\varepsilon\in\mathbb{R}_{>0}$ существуют такие $N\in \mathbb{N}_0$ и $\delta\in(0,1)$: $\forall n\in\mathbb{N}_0:$ $n\ge N$
    $$R(w,n,\varepsilon) \supseteq P(w,n,\delta).$$

Зафиксируем данные $N\in \mathbb{N}_0$ и $\delta\in\mathbb{R}_{>0}$.

По Теореме \ref{t1} при нашем $w\in\mathbb{YF}_\infty^+$ и только что зафиксированном $\delta\in\mathbb{R}_{>0}$
$$\lim_{n \to \infty}{\sum_{v\in P\left(w,n,\delta\right)}\mu_w(v)=1}.$$

Кроме того, ясно что

\begin{itemize}
    \item $\forall w\in\mathbb{YF}_\infty$ и $v\in\mathbb{YF}$
    $$\mu_w(v)=\lim_{m \to \infty} \frac{d(\varepsilon,v)d(v,w_m)}{d(\varepsilon,w_m)}\ge 0 ;$$
    \item $\forall w\in\mathbb{YF}_\infty,$ $n\in\mathbb{N}_0$ и $\varepsilon\in \mathbb{R}_{>0}$
    $$R(w,n,\varepsilon)=\left\{v\in\mathbb{YF}_n:\; \pi(v)\in(\pi(w)(1-\varepsilon),\pi(w)(1+\varepsilon)) \right\}\subseteq\mathbb{YF}_n;$$
    \item (Следствие \ref{mera1}) $\forall w\in\mathbb{YF}_\infty$ и $n\in\mathbb{N}_0$
    $${\sum_{v\in \mathbb{YF}_n}\mu_w(v)}=1.$$

\end{itemize}

А значит, при наших $w\in\mathbb{YF}_\infty^+,$ $\varepsilon\in\mathbb{R}_{>0}$ $\forall n\in\mathbb{N}_0:$ $n\ge N$
$${\sum_{v\in P(w,n,\delta)}\mu_w(v)}\le {\sum_{v\in R(w,n,\varepsilon)}\mu_w(v)} \le {\sum_{v\in \mathbb{YF}_n}\mu_w(v)}=1.$$

А значит, по Лемме о двух полицейских, 
$$\lim_{n \to \infty}{\sum_{v\in R(w,n,\varepsilon)}\mu_w(v)}=1,$$
что доказывает второй пункт.

Также заметим, что
\begin{itemize}
    \item $$\overline{R}\left(w,n,\varepsilon\right)\cup R\left(w,n,\varepsilon\right)=$$
    $$=\left\{v\in\mathbb{YF}_n: \pi(v)\notin(\pi(w)(1-\varepsilon),\pi(w)(1+\varepsilon)) \right\}
    \cup \left\{v\in\mathbb{YF}_n: \pi(v)\in(\pi(w)(1-\varepsilon),\pi(w)(1+\varepsilon)) \right\}
    =\mathbb{YF}_n;$$
    \item $$\overline{R}\left(w,n,\varepsilon\right)\cap R\left(w,n,\varepsilon\right)=$$
    $$=\left\{v\in\mathbb{YF}_n: \pi(v)\notin(\pi(w)(1-\varepsilon),\pi(w)(1+\varepsilon)) \right\}\cap \left\{v\in\mathbb{YF}_n: \pi(v)\in(\pi(w)(1-\varepsilon),\pi(w)(1+\varepsilon)) \right\}
    =\varnothing;$$
    \item (Следствие \ref{mera1}) $\forall w\in\mathbb{YF}_\infty$ и $n\in\mathbb{N}_0$
    $${\sum_{v\in \mathbb{YF}_n}\mu_w(v)}=1.$$
\end{itemize}
Таким образом, можно сделать вывод, что
$$\lim_{n \to \infty}{\sum_{v\in \overline{R}(w,n,\varepsilon)}\mu_w(v)}=0,$$
что доказывает первый пункт.

Таким образом, оба пункта доказаны.

Следствие доказано.

\end{proof}
\newpage

\section{Благодарности}

\begin{itemize}
    \item Работа поддержана грантом в форме субсидий из федерального бюджета на создание и развитие международных математических центров мирового уровня, соглашение между МОН и ПОМИ РАН № 075-15-2019-1620 от 8 ноября 2019 г., а также грантом фонда поддержки теоретической физики и математики "БАЗИС", договор No 19-7-2-39-1 от 1 сентября 2019 г.
    \item Мы признательны Фёдору Владимировичу Петрову за постановку задачи, помощь в публикации статьи и моральную поддержку на протяжении всего периода работы, а также Павлу Андреевичу Ходунову за проявленное при проверке доказательства терпение.
\end{itemize}

\end{document}